\documentclass[12pt]{elsarticle}
\usepackage{amsmath, amssymb, amsthm}
\usepackage{graphicx,amssymb,amsfonts,amsbsy,amsmath,amsthm}
\usepackage{mathtools}
\usepackage{hyperref}
\usepackage[utf8]{inputenc}
\usepackage{floatrow}
\usepackage{float}
\usepackage{subfig, tikz}
\usepackage{amssymb}
\usepackage{amsthm}
\usepackage{epstopdf}
\usepackage{anysize}
\usepackage{color,soul}
\usepackage{graphicx}
\usepackage{multirow}
\newtheorem{theorem}{Theorem}[section]
\newtheorem{lemma}[theorem]{Lemma}

\newtheorem{proposition}[theorem]{Proposition}

\DeclareMathOperator{\e}{e}

\newcommand{\greensqr}{\textcolor{green}{\rule{0.8em}{0.8em}}}
\newcommand{\redsqr}{\textcolor{red}{\rule{0.8em}{0.8em}}}
\newcommand{\bluesqr}{\textcolor{blue}{\rule{0.8em}{0.8em}}}

\journal{~}
\input epsf.sty
\usepackage{lineno}
\begin{document}

\begin{frontmatter}
\title{Double Allee effect induced extinction and bifurcation in a discrete predator-prey model}

\author[label1]{Rajesh Ranjan Patra\corref{cor1}}
\ead{rajeshp911@yahoo.com}
\author[label1]{Sarit Maitra}

\address[label1]{Department of Mathematics, NIT Durgapur, Durgapur-713209, India}
\cortext[cor1]{Corresponding author}

\begin{abstract}
The importance of the Allee effect in studying extinction vulnerability is widely recognized by researchers, and neglecting it could adversely impact the management of threatened or exploited populations \cite{berec2007}. In this article, we examine a discrete predator-prey model where the prey population is associated with two component Allee effects. We derive sufficient conditions for the existence and local stability nature of the fixed points of the system. The occurrence of Neimark-Sacker bifurcation is established, and sufficient conditions are obtained along with the normal form. Numerically, we demonstrate that the system exhibits Neimark-Sacker bifurcation for various system parameters. Additionally, the numerical simulations indicate that certain system parameters have threshold values, above or below which the populations are driven to extinction due to the effect of the double Allee effect.

\begin{keyword}
Holling type II \sep Double Allee effect \sep Neimark-Sacker bifurcation \sep Extinction
\end{keyword}

\end{abstract}
\end{frontmatter}

\section{Introduction}
Interspecific relations are abundant in nature, and they may significantly contribute to the fitness of interacting species. One such phenomenon that is associated with the individual fitness of a population is the Allee effect, which is widely popular among ecologists. It establishes a direct relationship between population density and individual fitness, which was first observed by W. C. Allee \cite{allee1931} in an aquatic system. The positive density dependence influences the growth rate of a population when there is an increase in the population density, as observed in goldfish (\emph{Carassius auratus}) population \cite{allee1931} and Argentine ants (\emph{Linepithema humile}) \cite{giraud2002,anuglo2018}. This correlation becomes more relevant in low population levels since it can lead to a decline in population growth rates, as observed in desert bighorn sheep \cite{mooring2004} and African wild dogs \cite{courchamp2002}. It also suggests that there is a minimum threshold population size below which a species may struggle to survive or reproduce. Hunting cooperation, mate-finding ability, and enhanced protection from predators in larger populations are some of the various factors that contribute to the Allee effect \cite{kuussaari1998, gascoigne2009, alves2017}. 

In the context of the reproduction-predation risk trade-off, both reproductive success and the probability of avoiding predation may show a positive correlation with population size or density, which gives rise to a component Allee effect \cite{courchamp1999, pavlova2010}. Allee effects at the component level are evident in individual fitness components, such as juvenile survival or litter size. In contrast, demographic Allee effects operate at the broader scale of overall mean individual fitness, typically observed through the population's demography, specifically the per capita population growth rate \cite{courchamp2008}. A demographic Allee effect can be categorized as either a strong or a weak Allee effect. In a strong Allee effect, the growth rate of a population becomes negative when its density falls below a certain threshold. However, in a weak Allee effect, the growth rate always stays positive, even at low population levels, but at a reduced rate. There are several instances of strong Allee effect in plant species \cite{groom1998}, insects \cite{berggren2001} and vertebrates \cite{angulo2007}. Examples of weak Allee effect include Smooth cordgrass (\emph{Spartina alterniflo}) \cite{taylor2004}. Other examples of Allee effects can be found in \cite{berec2007, kuussaari1998, kramer2009}

Allee effects are typically examined as single mechanisms of positive density dependence; however, there can be multiple Allee effects present in a single ecosystem. Certain examples include Island fox (\emph{Urocyon littoralis}) in  California Channel Islands \cite{angulo2007} and Alpine marmot (\emph{Marmota marmota}) \cite{stephens2002}. Many times, the interaction between these Allee effects can be complex \cite{berec2007, li2022}. Researchers combine two (or more) types of Allee effects to study its impact on the population dynamics. This is termed the double (or multiple) Allee effect. Historically, there have been various studies that examine the mathematical aspects of Allee effects on the population dynamics. 
In a predator-prey model with a double Allee effect, Feng and Kang \cite{feng2015} demonstrated that, while a strong Allee effect leads to the emergence of a basin of attraction for total extinction equilibrium, both populations can survive under a weak Allee effect. 
Considering one Allee effect on prey and one on predator, Xing et al. \cite{xing2023} obtained that the survival of prey is positively influenced by a strong Allee effect on the predator. 
These studies impose a component Allee effect on prey and predator species. 

The dynamics of a population with a single Allee effect is generally modelled by the well-known model \cite{kot2001,naik2022}
\begin{eqnarray*} 
\frac{dx}{dt}=rx\left(1-\frac{x}{K}\right)(x-m),
\end{eqnarray*}
where $0<m\leq K$ represents a strong Allee effect and $-K< m\leq 0$ represents a weak Allee effect \cite{gonzalez2011}. here, $r$ is the intrinsic growth rate and $K$ is the carrying capacity for the population $(x)$. As shown in \cite{zu2010}, sometimes the factor $\frac{rx}{x+n}$ is associated with the per capita growth term to represent Allee effect in the population. This term is introduced to deal with the effect of fertility rate on the growth of the population. Also, there are several other mathematical forms that have previously been used to model the Allee effect in a population \cite{boukal2002}.

In the last few decades, several researchers have studied the influence of double Allee effect in a single population, incorporating these two factors 
%
%
in the prey population in continuous and discrete models \cite{boukal2007, gonzalez2015, tiwari2021, xia2022}. 
A study by Boukal et al. \cite{boukal2007} shows that abrupt and deterministic collapses in the system can occur without prior fluctuations in predator–prey dynamics, typically observed in scenarios with steep type III functional responses and strong Allee effects. 
Gonzalez-Olivares \cite{gonzalez2015} studied a biomass-effort fishery model and obtained that the presence of double Allee effect in the fish population induces three potential attractors in the phase plane. This makes the population sensitive to disturbances, and this needs to be managed carefully. 
Mathematically, the population dynamics with double Allee effect is represented by
\begin{equation*} 
\frac{dx}{dt} = \frac{rx}{x+n}\left(1-\frac{x}{K}\right)(x-m).
\end{equation*}
The expression $\frac{rx}{x+n}$ can be seen as an approximation of population dynamics where distinctions between fertile and non-fertile individuals are not explicitly modelled. We can infer that this term signifies the influence of the Allee effect attributable to the non-fertile population $n$ \cite{courchamp2008, gonzalez2011}. Although, as stated earlier, the evidence of double Allee effects is not rare in nature, still studies on double Allee effect are far lesser compared to that of single Allee effect.

The article is organised as follows. The model formulation is described in Section \ref{sec_model}. Existence and local stability analysis of the possible fixed points of the considered predator-prey model is presented in Section \ref{sec_prelim}. Section \ref{sec_bif} includes the bifurcation analysis of the system for the system parameters. Numerical simulations illustrating the occurrence of bifurcation along with the phase portraits are presented in Section \ref{sec_num}. Finally, the article concludes in Section \ref{sec_con}.
%
%
\section{The model}\label{sec_model}
In this section, we formulate a continuous-time two-dimensional predator-prey model where the predator-prey interaction term is represented by a Holling type II functional response, and the prey population has a logistic growth term with double Allee effect. Let $N$ and $P$ denote the density of the prey and predator population, respectively, then the mathematical model representing the population dynamics is given by
\begin{align}\label{pre_sys5}
\frac{dN}{dT} &= rN\left(1-\frac{N}{K}\right)\left(\frac{N-p}{N+q}\right)-\frac{aNP}{1+ahN},\nonumber\\
\frac{dP}{dT} &= \frac{baNP}{1+ahN}-cP.
\end{align}
Let $x=N/K$, $y=aP/r$ \& $t=rT$, then system \eqref{pre_sys5} is transformed to the following system
\begin{align}\label{cts_sys5}
\frac{dx}{dt} &= x(1-x)\left(\frac{x-s}{x+w}\right)-\frac{xy}{1+\alpha x},\nonumber\\
\frac{dy}{dt} &= \frac{\beta xy}{1+\alpha x}-\theta y,
\end{align}
where $s=p/K$, $w=q/K$, $\alpha=ahK$, $\beta=baK/r$ \& $\theta=c/r$. 
It is suggested that discrete models are more realistic than continuous analogues because biological sample data statistics are collected at specific time intervals \cite{zhou2023}. 
Also, discrete-time dynamical systems are best suited to model nonoverlapping generations \cite{may1974} and show much more complex dynamics than the continuous-time dynamical systems \cite{may1976, neubert1992, yousef2018}. 
Many discrete-time models are obtained from continuous systems through the application of the forward Euler scheme \cite{liu2007,ren2016} or the method of semi-discretization \cite{xia2022, zhou2023, wang2015}. 
Converting this system into discrete system, as per \cite{xia2022}, by integrating these equations on the interval $[n,t)$ after semi-discretization and for $t\rightarrow n+1$, we have
\begin{align}\label{sys5}
x_{n+1} &= x_n\;\exp\left[(1-x_n)\left(\frac{x_n-s}{x_n+w}\right)-\frac{y_n}{1+\alpha x_n}\right],\nonumber\\
y_{n+1} &= y_n\; \exp\left[\frac{\beta x_n}{1+\alpha x_n}-\theta\right],
\end{align}
where $s$ is the threshold for the Allee effect, $w$ affects the overall per capita growth rate of the prey, $\alpha$, $\beta$ is the conversion rate of prey into predator biomass and $\theta$ is the natural death rate of the predator.

\section{Preliminary analysis}\label{sec_prelim}
\subsection{Fixed points}
We observe that the system has a maximum of up to four fixed points, namely $E_0(0,0)$, $E_s(s,0)$, $E_1(1,0)$ and $E_+(x_+,y_+)$. $E_s$ exists when the system shows a strong Allee effect, i.e., $s>0$, and it vanishes in case of weak Allee effect $(s<0)$. Here, $$x_+=\frac{\theta}{\beta-\alpha\theta}\quad \text{and}\quad y_+=\frac{\beta}{\beta-\alpha\theta}\left(1-\frac{\theta}{\beta-\alpha\theta}\right)\frac{\theta-s(\beta-\alpha\theta)}{\theta+w(\beta-\alpha\theta)}.$$ Hence, $E_+$ exists when $\beta>\alpha\theta$ holds along with $s<\frac{\theta}{\beta-\alpha\theta}<1$ or $1<\frac{\theta}{\beta-\alpha\theta}<s$.
\subsection{Local stability analysis}
Let $x_{n+1}=f(x_n,y_n)$ and $y_{n+1}=g(x_n,y_n)$. Then the Jacobian for system \eqref{sys5} is given by
\begin{equation} \label{jac_sys}
J(x,y)=\frac{\partial (f,g)}{\partial (x,y)}=\begin{pmatrix}
1+\left(-1+\frac{s+w+sw+w^2}{(x+w)^2}=\frac{\alpha y}{(1+\alpha x)^2}\right)f & -\frac{y}{1+\alpha x}f\\
\frac{\beta}{(1+\alpha x)^2}g & \frac{1}{y}g
\end{pmatrix}.
\end{equation}
Substituting the fixed points in the Jacobian matrix, we can analyse their local stability by analysing the characteristic roots of the corresponding matrix.
\begin{proposition}
The fixed point $E_0$ is a
\begin{enumerate}[(i)]
\item sink for $s>0$;
\item saddle point for $s<0$;
\item non-hyperbolic point for $s=0$.
\end{enumerate}
\end{proposition}
\begin{proof}
The Jacobian in \eqref{jac_sys} evaluated at $E_0$ is given by 
\[J\left.\middle|_{E_0}\right. = \begin{pmatrix} \e^{-s/w} & 0\\ 0 & \e^{-\theta} \end{pmatrix} .\]
Then the eigenvalues of the matrix are $\lambda_1=\e^{-s/w}$ \& $\lambda_2=\e^{-\theta}$. Therefore, $E_0$ is a sink for $s>0$, a saddle for $s<0$ and a non-hyperbolic fixed point for $s=0$. Moreover, it will never be a source.
\end{proof}

\begin{proposition}
Let $\beta>\alpha\theta$. Then $E_1$ is a
\begin{enumerate}[(i)]
\item sink if $s<1<\frac{\theta}{\beta-\alpha\theta}$;
\item source if $s>1>\frac{\theta}{\beta-\alpha\theta}$;
\item saddle point if $\min\left\{s,\frac{\theta}{\beta-\alpha\theta}\right\}>1$ or $\max\left\{s,\frac{\theta}{\beta-\alpha\theta}\right\}<1$;
\item non-hyperbolic point if $s=1$ or $\frac{\theta}{\beta-\alpha\theta}=1$ or $s=\frac{\theta}{\beta-\alpha\theta}=1$.
\end{enumerate}
\end{proposition}
\begin{proof}
For $\beta>\alpha\theta$, the Jacobian in \eqref{jac_sys} evaluated at $E_1$ is given by
\[J\left.\middle|_{E_1}\right. =
\begin{pmatrix}
1+\frac{s-1}{s+w} & -\frac{1}{\alpha+1}\\
0 & \e^{\frac{\beta}{\alpha+1}-\theta}.
\end{pmatrix}
\]
Then the eigenvalues of the Jacobian are $\lambda_1=1-\frac{s-1}{s+w}$ \& $\lambda_2 = \e^{\frac{\beta}{\alpha +1}-\theta}$. For $s=1$, $|\lambda_1|=1$; for $s>1$, $|\lambda_1|>1$; and for $s<1$, the value of $|\lambda_1|$ will depend on the root, $s_0$, of the equation 
$$\frac{s(s-1)}{s+w}=2,\text{ where } s_0=2w+3.$$
So, for the case $s<1$, we will have $|\lambda_1|<1$ when $s<s_0$; $|\lambda_1|>1$ when $s>s_0$; and $|\lambda_1|=1$ when $s=s_0=2w+3$. 
Clearly, $|\lambda_1|\ngeqslant1$ for $s<1$ since $s_0=2w+3>3$. 
Hence, $|\lambda_1|<1$ when $s<1$. 
Also, we can have $|\lambda_2|=1$ when $\frac{\theta}{\beta-\alpha\theta}=1$; $|\lambda_2|<1$ when $\frac{\theta}{\beta-\alpha\theta}>1$; and $|\lambda_2|>1$ when $\frac{\theta}{\beta-\alpha\theta}<1$. 
Considering the cases for $|\lambda_{1,2}|<1$, we have $\frac{\theta}{\beta-\alpha\theta}>1$ and $s<1$. 
Thus we can simply say that $E_1$ is a sink when $s<1<\frac{\theta}{\beta-\alpha\theta}$. 
Similarly, we can determine the conditions under which $E_1$ is a source, saddle, or non-hyperbolic fixed point by considering different cases based on the absolute values of $\lambda_1$ and $\lambda_2$.
\end{proof}
\begin{figure}[H]
 \subfloat[$\beta>\alpha\theta$ \& $\theta/(\beta-\alpha\theta)<1$ for $\alpha=8.4$\label{subfig-6}]{%
  \includegraphics[width=0.45\textwidth, height=6cm]{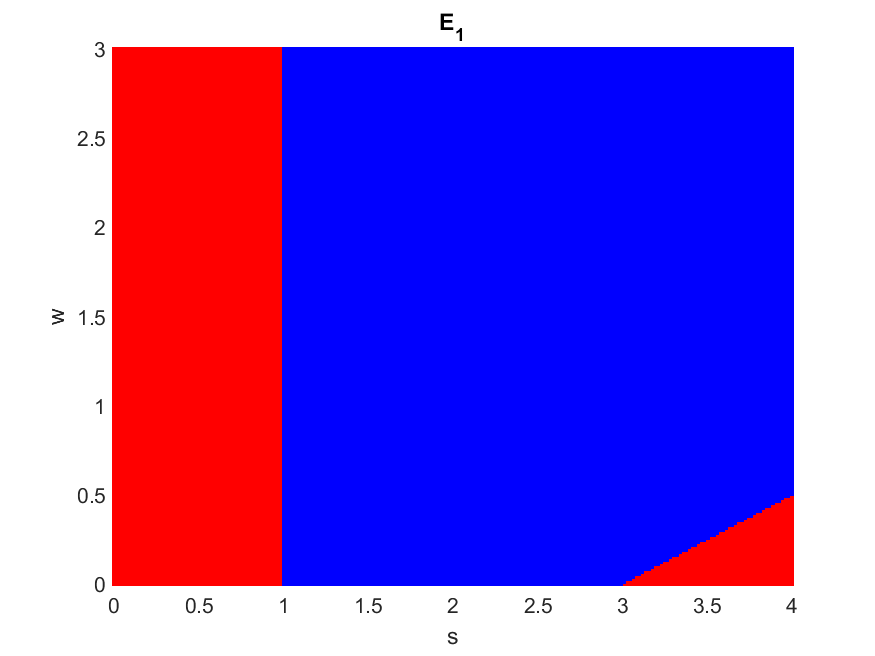}}
  \subfloat[$\beta>\alpha\theta$ \& $\theta/(\beta-\alpha\theta)>1$ for $\alpha=9.4$\label{subfig-6}]{%
  \includegraphics[width=0.45\textwidth, height=6cm]{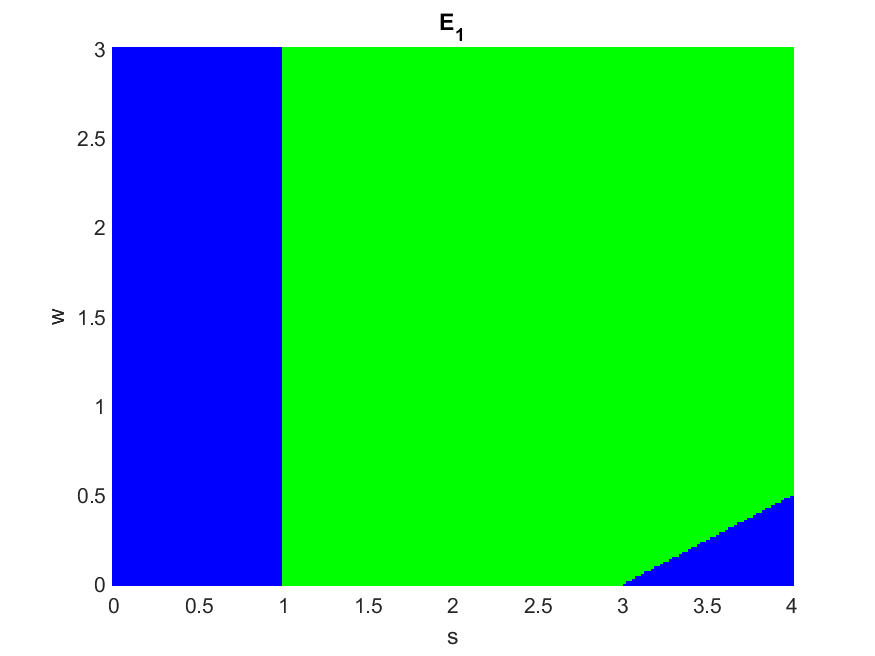}
  }
 \caption{Topological classification of $E_1(1,0)$ for $\beta=1.3$, $\theta=0.13$, $s\in [0,4]$, $w\in [0,3]$ and $\alpha\in\{8.4,9.4\}$. 
 $[~\vcenter{\hbox{$\greensqr$}}$ : sink; ~$\vcenter{\hbox{$\redsqr$}}$ : source; ~$\vcenter{\hbox{$\bluesqr$}}$ : saddle$~]$}
 \label{fig:e1_1}
\end{figure}

\begin{proposition}
Let $\beta=\alpha\theta$. Then $E_1$ is a
\begin{enumerate}[(i)]
\item sink if $-2w-1<s<1$;
\item saddle point if $s<-2w-1$ or $s>1$;
\item non-hyperbolic point if $s=-2w-1$ or $s=1$.
\end{enumerate}
\end{proposition}
\begin{proof}
For $\beta=\alpha\theta$, the Jacobian at $E_1$ becomes
\[J \left.\middle|_{E_1}\right. =\begin{pmatrix}
1+\frac{s-1}{w+1} & -\frac{1}{\alpha+1}\\
0 & \e^{-\frac{\theta}{\alpha+1}}
\end{pmatrix}. \]
Then the eigenvalues of the Jacobian are $\lambda_1=1+\frac{s-1}{w+1}$ and $\lambda_2=\e^{-\frac{\theta}{\alpha+1}}$. Clearly, $|\lambda_2|<1$. So, the nature of $E_1$ depends on the absolute value of $\lambda_1$. We can now easily obtain the required conditions as given in the proposition.
\end{proof}

\begin{figure}[H]
\includegraphics[width=0.45\textwidth, height=6cm]{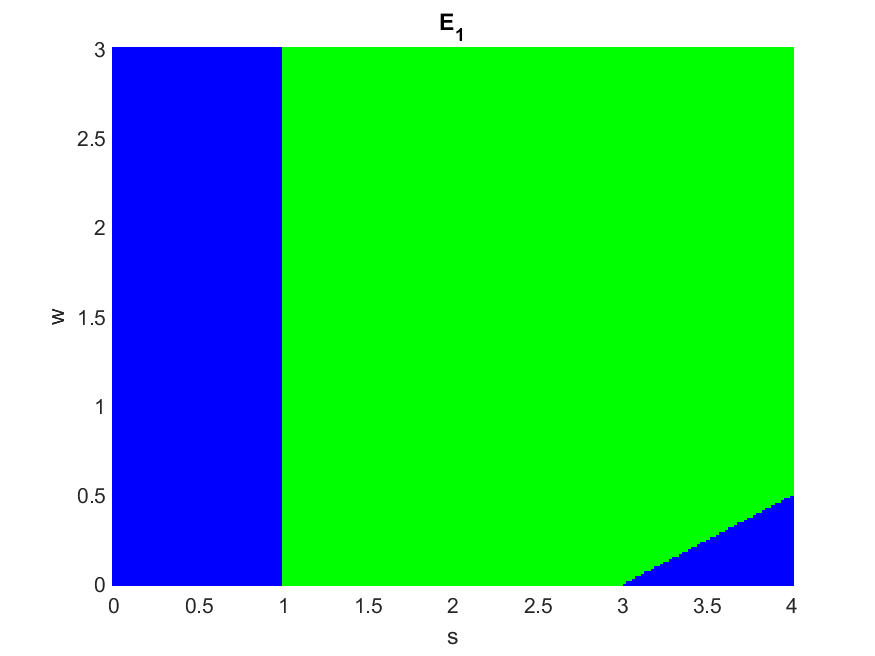}
\caption{Topological classification of $E_s(s,0)$ when $\beta=\alpha\theta$ for $\alpha=10$ for $\beta=1.3$, $\theta=0.13$, $s\in [0,4]$, $w\in [0,3]$ and $\alpha=10$. 
$[~\vcenter{\hbox{$\greensqr$}}$ : sink; ~$\vcenter{\hbox{$\bluesqr$}}$ : saddle$~]$}
\label{fig:e1_1}
\end{figure}

\begin{proposition}
Let $s>0$ and $\beta>\alpha\theta$. Then, the fixed point $E_s$ is a
\begin{enumerate}[(i)]
\item sink if $1<s<\min\left\{\frac{\theta}{\beta-\alpha\theta},s_+\right\}$;
\item source if $\frac{\theta}{\beta-\alpha\theta}<s<1$ or $s>\max\left\{\frac{\theta}{\beta-\alpha\theta},s_+\right\}$;
\item saddle point if $s<\min\{1,\frac{\theta}{\beta-\alpha\theta}\}$ or $\max\left\{1,s_+\right\}<s<\frac{\theta}{\beta-\alpha\theta}$ or $\max\left\{1,\frac{\theta}{\beta-\alpha\theta}\right\}<s<s_+$;
\item non-hyperbolic point if $s=\frac{\theta}{\beta-\alpha\theta}=1$ or $s=s_+=\frac{\theta}{\beta-\alpha\theta}>1$,
\end{enumerate}
where $s_+=(3+\sqrt{8w+9})/2$ which is the largest positive root of $\frac{s(s-1)}{s+w}=2$.
\end{proposition}
\begin{proof}
Clearly, $s>0$ must hold for the existence of $E_s$. For $\beta>\alpha\theta$, the Jacobian in \eqref{jac_sys} evaluated at $E_s$ is given by
\[J\left.\middle|_{E_s}\right. =
\begin{pmatrix}
1-\frac{s(s-1)}{s+w} & -\frac{s}{\alpha s+1}\\
0 & \e^{\frac{\beta s}{\alpha s+1}-\theta}.
\end{pmatrix}
\]
Then the eigenvalues of the Jacobian are $\lambda_1=1-\frac{s(s-1)}{s+w}$ \& $\lambda_2 = \e^{\frac{\beta s}{\alpha s+1}-\theta}$. For $s=1$, $|\lambda_1|=1$; for $s<1$, $|\lambda_1|>1$; and for $s>1$, the value of $|\lambda_1|$ will depend on the positive root, $s_+$, of the equation $$\frac{s(s-1)}{s+w}=2,\text{ where } s_+=\frac{3+\sqrt{8w+9}}{2}>3.$$
So, for the case $s>1$, we will have $|\lambda_1|<1$ when $s\in (1,s_+)$; $|\lambda_1|>1$ when $s\in (s_+,\infty)$; and $|\lambda_1|=1$ when $s=s_+$. Similarly, we can obtain that $|\lambda_2|=1$ when $s=\frac{\theta}{\beta-\alpha\theta}$; $|\lambda_2|<1$ when $s<\frac{\theta}{\beta-\alpha\theta}$; and $|\lambda_2|>1$ when $s>\frac{\theta}{\beta-\alpha\theta}$. Combining the cases according to the absolute values of $\lambda_{1,2}$, we can conclude the statement of the proposition.
\end{proof}

\begin{figure}[H]
 \subfloat[$\beta>\alpha\theta$ \& $\theta/(\beta-\alpha\theta)<1$ for $\alpha=8.4$\label{subfig-6}]{%
  \includegraphics[width=0.45\textwidth, height=6cm]{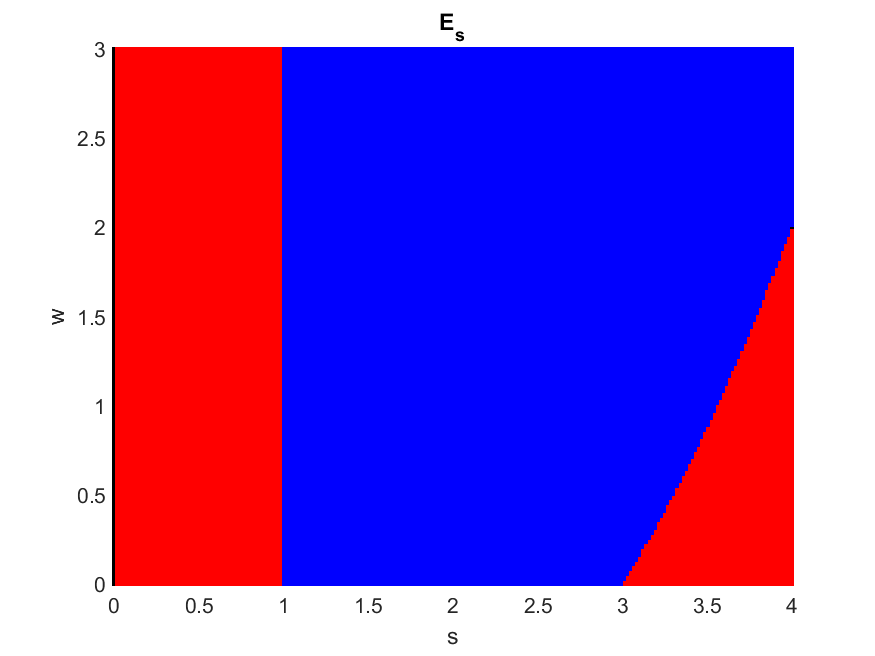}}
  \subfloat[$\beta>\alpha\theta$ \& $\theta/(\beta-\alpha\theta)>1$ for $\alpha=9.4$\label{subfig-6}]{%
  \includegraphics[width=0.45\textwidth, height=6cm]{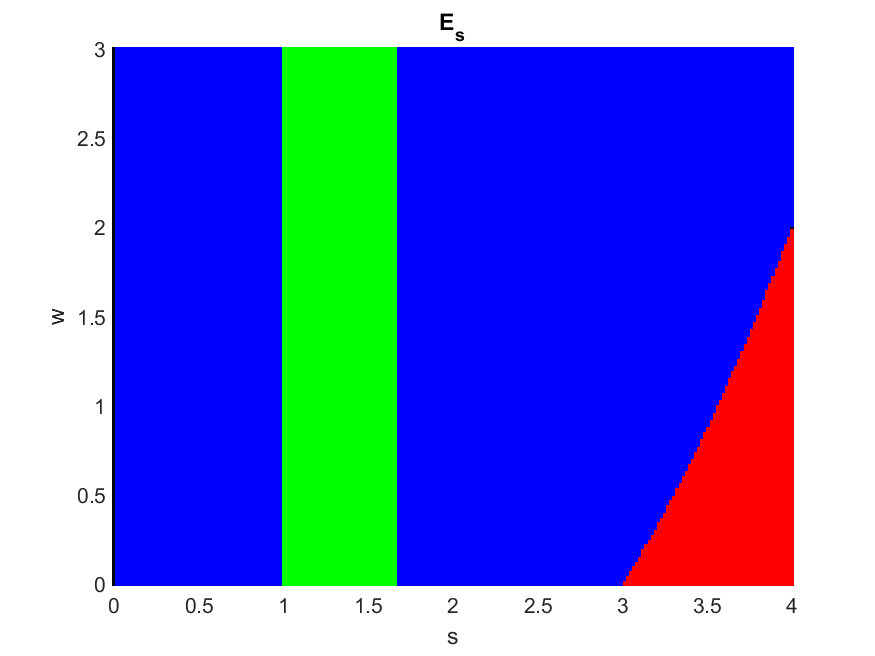}
  }
\caption{Topological classification of $E_s(s,0)$ for $\beta=1.3$, $\theta=0.13$, $s\in [0,4]$, $w\in [0,3]$ and $\alpha\in\{8.4,9.4\}$. 
$[~\vcenter{\hbox{$\greensqr$}}$ : sink; ~$\vcenter{\hbox{$\redsqr$}}$ : source; ~$\vcenter{\hbox{$\bluesqr$}}$ : saddle$]$}
 \label{fig:es_1}
\end{figure}

\begin{proposition}
Let $s>0$ and $\beta=\alpha\theta$. Then, the fixed point $E_s$ is a
\begin{enumerate}[(i)]
\item sink if $1<s<s_+$;
\item saddle point if $s<1$ or $s>s_+$;
\item non-hyperbolic point if $s=1$ or $s=s_+$. 
\end{enumerate}
\end{proposition}
\begin{proof}
For $\beta=\alpha\theta$, the Jacobian at $E_s$ becomes
\[J \left.\middle|_{E_1}\right. =\begin{pmatrix}
1-\frac{s-1}{w+1} & -\frac{s}{\alpha s+1}\\
0 & \e^{-\frac{\theta}{\alpha s+1}}
\end{pmatrix}. \]
Then the eigenvalues of the Jacobian are $\lambda_1=1-\frac{s-1}{w+1}$ and $\lambda_2=\e^{-\frac{\theta}{\alpha s+1}}$. Clearly, $|\lambda_2|<1$ since $s>0$. So, the nature of $E_1$ depends on the absolute value of $\lambda_1$. Now, it can be deduced that the sufficient conditions given in the propositions hold true.
\end{proof}

\begin{figure}[H]
  \includegraphics[width=0.45\textwidth, height=6cm]{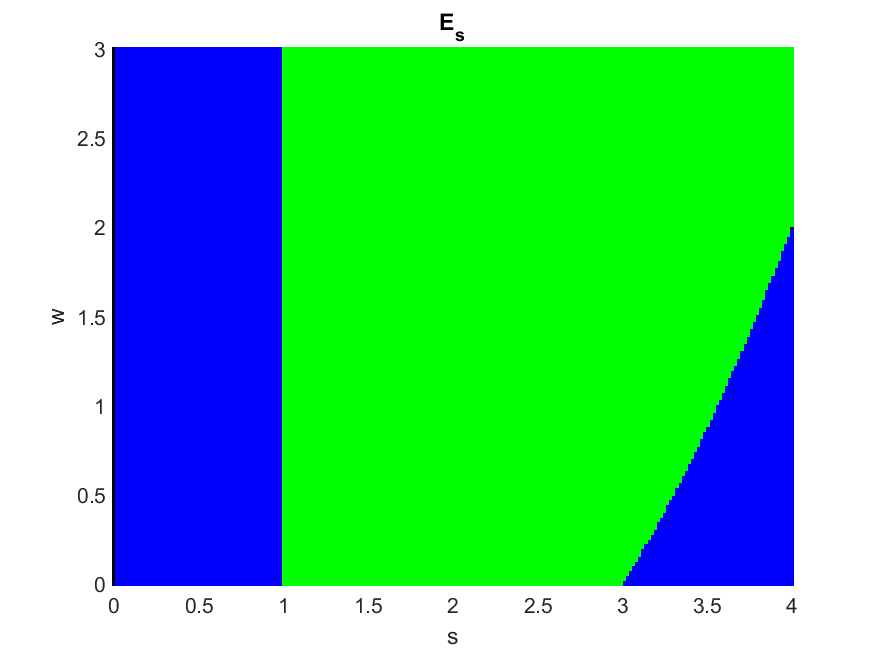}
 \caption{Topological classification of $E_s(s,0)$ when $\beta=\alpha\theta$ for $\alpha=10$ for $\beta=1.3$, $\theta=0.13$, $s\in [0,4]$, $w\in [0,3]$ and $\alpha=10$. 
 $[~\vcenter{\hbox{$\greensqr$}}$ : sink;  ~$\vcenter{\hbox{$\bluesqr$}}$ : saddle$~]$}
 \label{fig:es_2}
\end{figure}

At the positive fixed point $E_+(x_+,y_+)$, the variational matrix of system \eqref{sys5} is given by
\begin{equation*}
J(E_+)=\begin{pmatrix}
J_{11} & J_{12}\\
J_{21} & J_{22}
\end{pmatrix},
\end{equation*}
where $$J_{11}=1-x_++\frac{(s+w+sw+w^2)x_+}{(x_++w)^2}+\frac{\alpha x_+y_+}{(1+\alpha x_+)^2},$$ $$J_{12}=-\frac{x_+}{1+\alpha x_+},\; J_{21}=\frac{\beta y_+}{(1+\alpha x_+)^2}\;\; \&\;\; J_{22}=1.$$
Therefore, we can evaluate the following quantities
\begin{align*}
\text{tr}(J) &= 2-x_++\frac{(s+w+sw+w^2)x_+}{(x_++w)^2}+\frac{\alpha x_+y_+}{(1+\alpha x_+)^2},\\
\text{det}(J) &= 1-x_++\frac{(s+w+sw+w^2)x_+}{(x_++w)^2}+\frac{\alpha x_+y_+}{(1+\alpha x_+)^2}+\frac{\beta x_+y_+}{(1+\alpha x_+)^3}.
\end{align*}
\begin{lemma}[\cite{din2016,isik2019}]
Let $F(\lambda)=\lambda^2+B\lambda+C$ is the characteristic polynomial of the Jacobian at $E_+$ and let $\lambda_1$ and $\lambda_2$ are two roots of $F(\lambda)=0$. Since $F(1)>0$, so the following statements hold.
\begin{enumerate}[(a)]
\item $E_+$ is a sink, i.e., $|\lambda_1|<1$ \& $|\lambda_2|<1$ iff $F(-1)>0$ and $C<1$.
\item $E_+$ is a source, i.e., $|\lambda_1|>1$ \& $|\lambda_2|>1$ iff $F(-1)>0$ and $C>1$.
\item $E_+$ is a saddle point, i.e., $|\lambda_1|<1$ \& $|\lambda_2|>1$ iff $F(-1)<0$.
\item $E_+$ is non-hyperbolic, i.e., $\lambda_1$ \& $\lambda_2$ are complex conjugates and $|\lambda_1|=|\lambda_2|=1$ iff $-2<B<2$ \& $C=1$.
\item $E_+$ is non-hyperbolic, i.e., $\lambda_1=\lambda_2=-1$ iff $F(-1)=0$ and $B=2$.
\item $E_+$ is non-hyperbolic, i.e., $\lambda_1=-1$ \& $\lambda_2\neq -1$ iff $ F(-1)=0$ and $B\neq 2$.
\end{enumerate}
\end{lemma}

\begin{figure}[H]
 \subfloat[$s=0.0125$, $w=0.125$, $\theta=0.13$\label{subfig-4}]{%
  \includegraphics[width=0.45\textwidth, height=6cm]{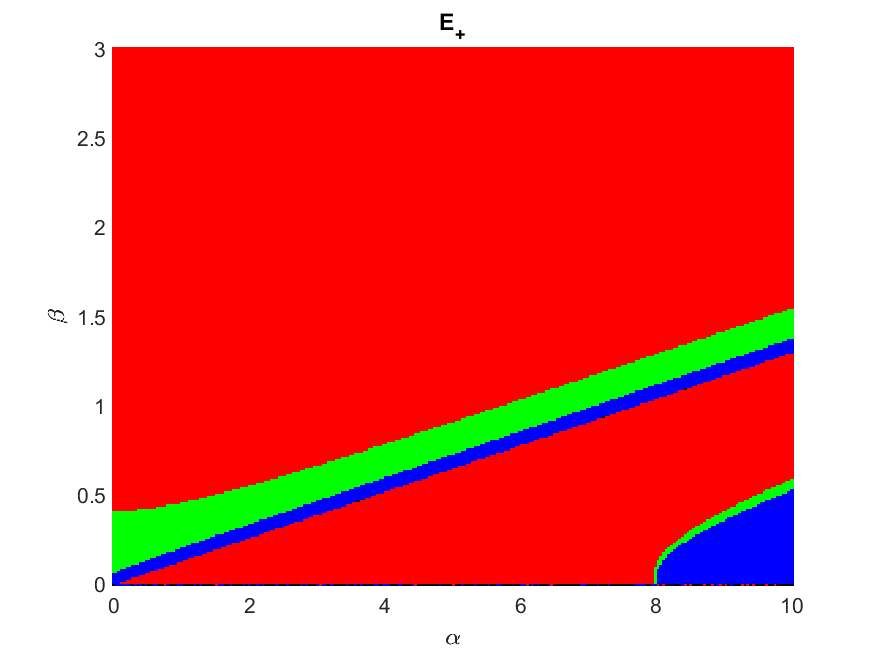}}
 \subfloat[$\alpha=8.4$, $\beta=1.3$, $\theta=0.13$\label{subfig-5}]{%
  \includegraphics[width=0.45\textwidth, height=6cm]{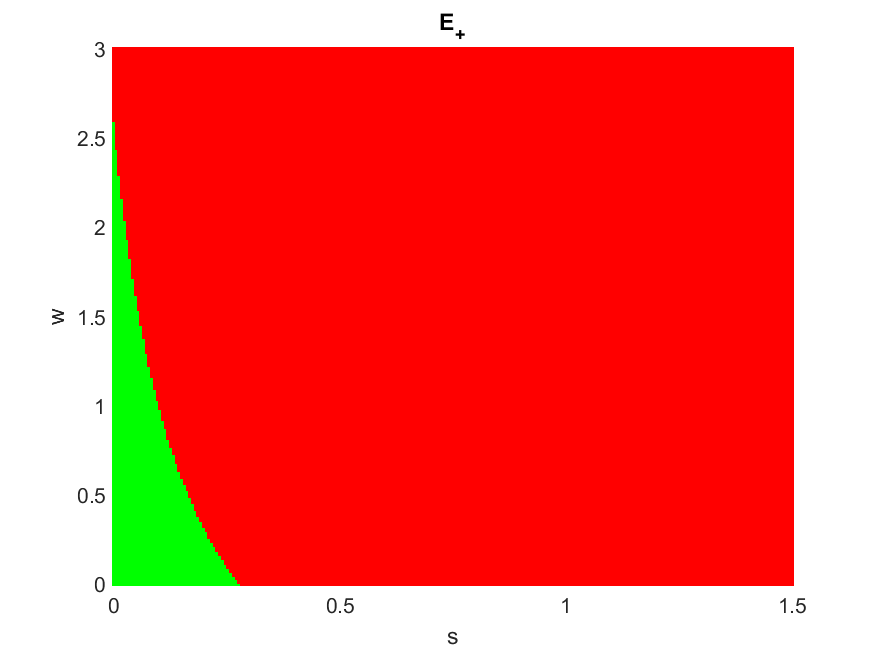}
  }
 \caption{Topological classification of $E_+(x_+,y_+)$ for. 
 $[~\vcenter{\hbox{$\greensqr$}}$ : sink; ~$\vcenter{\hbox{$\redsqr$}}$ : source; ~$\vcenter{\hbox{$\bluesqr$}}$ : saddle$~]$}
\end{figure}
\section{Bifurcation analysis}\label{sec_bif}
The Neimark-Sacker bifurcation occurs in a discrete system when a closed invariant curve is born from an equilibrium point, signalling a change in stability through a pair of complex eigenvalues with unit modulus. Utilizing the center manifold theorem and bifurcation theory \cite{xia2022, yousef2018, kuznetsov2004}, we systematically derive the normal form and establish the necessary conditions for Neimark-Sacker bifurcation. The parameter $\beta$ is considered to be the bifurcation parameter. From the Jacobian of system \eqref{sys5} near the positive fixed point, the characteristic polynomial is given by


\begin{equation*}
F(\lambda)=\lambda^2-\text{tr}(J)\lambda+\text{det}(J).
\end{equation*}
For Neimark-Sacker bifurcation, required conditions are $$\left(\text{tr}(J)\right)^2-4\;\text{det}(J)<0\;\;\&\;\; \text{det}(J)=1.$$
Hence, system \eqref{sys5} undergoes Neimark-Sacker (NS) bifurcation at $\beta=\beta_{NS}$, where $\beta$ varies in a small neighbourhood $\Omega$ defined by
$$\Omega=\left\{(s,w,\alpha,\beta,\theta):\left(\text{tr}(J)\right)^2-4\;\text{det}(J)<0,\;\text{det}(J)=1,\;\beta>\alpha\theta,\;\frac{\theta}{\beta-\alpha\theta}\neq s,\;\frac{\theta}{\beta-\alpha\theta}\neq 1\right\}.$$
So the system at $\beta=\beta_{NS}$ becomes
\begin{align}\label{sys_beta_ns}
x_{n+1} &= x_n\;\exp\left[(1-x_n)\left(\frac{x_n-s}{x_n+w}\right)-\frac{y_n}{1+\alpha x_n}\right],\nonumber\\
y_{n+1} &= y_n\;\exp\left[\frac{\beta_{NS}x_n}{1+\alpha x_n}-\theta\right].
\end{align}
Next, we assume a small perturbation of system \eqref{sys_beta_ns} by selecting \(\beta_*\) as the perturbation parameter in the following manner
\begin{align*}
x_{n+1} &= x_n\;\exp\left[(1-x_n)\left(\frac{x_n-s}{x_n+w}\right)-\frac{y_n}{1+\alpha x_n}\right],\nonumber\\
y_{n+1} &= y_n\;\exp\left[\frac{(\beta_{NS}+\beta_*)x_n}{1+\alpha x_n}-\theta\right].
\end{align*}
Let $u_n=x_n-x_+$, $v_n=y_n-y_+$, then the positive fixed point is shifted to the origin and the corresponding system is
\begin{align}\label{unvn_sys}
u_{n+1} &= (u_n+x_+)\;\exp\left[(1-u_n-x_+)\left(\frac{u_n+x_+-s}{u_n	x_+	w}\right)-\frac{v_n+y_+}{1+\alpha(u_n+x_+)}\right],\nonumber\\
v_{n+1} &= (v_n+y_+)\;\exp\left[\frac{(\beta_{NS}+\beta_*)(u_n+x_+)}{1+\alpha (u_n+x_+)}-\theta\right].
\end{align}
The characteristic equation associated with this model at $(0,0)$ is given by
\begin{equation*}
\lambda_2+M_1(\beta_*)\;\lambda+M_2(\beta_*)=0,
\end{equation*}
where $M_1(\beta_*)=-\text{tr}(J)|_{\beta=\beta_{NS}+\beta_*}$ and $M_2(\beta_*)=\text{det}(J)|_{\beta=\beta_{NS}+\beta_*}$. The roots of this characteristic equation are
\begin{equation*}
\lambda_{1,2}=\frac{1}{2}\left[-M_1(\beta_*)\pm i \sqrt{4M_2(\beta_*)-M_1^2(\beta_*)}\right],\;\; |\lambda_{1,2}|=M_2(\beta_*).
\end{equation*}
Now, we assume that
\begin{equation}\label{bif_cond_1}
\frac{d}{d\beta_*}|\lambda_{1,2}|_{\beta_*=0}\neq 0.
\end{equation}
According to the theory of Neimark-Sacker (NS) bifurcation \cite{kuznetsov2004}, it is required that at $\beta_*=0$, $\lambda_{1,2}^m\neq 1$ for $m=1,2,3,4$. This leads us to 
\begin{equation}\label{bif_cond_2}
M_1(0)\neq -2,0,1,2.
\end{equation}
The Taylor's series expansion of \eqref{unvn_sys} up to terms of order 3 produces the following model
\begin{align}
u_{n+1} &= s_{10}u_n+s_{01}v_n+s_{20}u_n^2+s_{11}u_nv_n+s_{02}v_n^2+s_{30}u_n^3+s_{21}u_n^2v_n+s_{12}u_nv_n^2+s_{03}v_n^3+O\left(\rho\right),\nonumber\\
v_{n+1} &= t_{10}u_n+t_{01}v_n+t_{20}u_n^2+t_{11}u_nv_n+t_{02}v_n^2+t_{30}u_n^3+t_{21}u_n^2v_n+t_{12}u_nv_n^2+t_{03}v_n^3+O\left(\rho\right),
\end{align}
where $\rho=(|u_n|+|v_n|)^3$, and
\begin{align*}
s_{10} &= 1 - x_+z_1, & t_{10} &= y_+z_2, \\
s_{01} &= \frac{x_+}{1+\alpha x_+}, & t_{01} &= 1,\\
s_{20} &= \frac{1}{2}x_+z_1^2-z_1-x_+z_3, & t_{20} &= -\frac{2\alpha\beta}{(1+\alpha x_+)^3} + \frac{1}{2}y_+z_2,\\
s_{11} &= -\frac{1}{(1+\alpha x_+)^2}-\frac{x_+z_1}{1+\alpha x_+}, & t_{11} &= z_2,\\
s_{02} &= \frac{x_+}{2(1+\alpha x_+)^2}, & t_{02} &= 0,\\
s_{03} &= \frac{x_+}{6(1+\alpha x_+)^3}, & t_{02} &= 0,\\
z_1 &= 1-\frac{(1+w)(s+w)}{(x_++w)^2}+\frac{\alpha y_+}{(1+\alpha x_+)^2},& z_2 &= \frac{\beta}{(1+\alpha x_+)^2},\\
z_3 &=\frac{(1+w)(s+w)}{(x_++w)^3}+\frac{\alpha^2y_+}{(1+\alpha x_+)^3}.& &
\end{align*}
Let $$J_1=\begin{pmatrix}
s_{10} & s_{01}\\
t_{10} & t_{01}
\end{pmatrix}.$$
Let the eigenvalues of the matrix be of the form $\lambda_{1,2}=\eta_1+i\eta_2$ with $|\eta_1+i\eta_2|=1$. Thus $$2\eta_1=s_{10}+t_{01}=-M(0)\quad \& \quad \eta_1^2+\eta_2^2=s_{10}t_{01}-s_{01}t_{10}=1.$$
Now, let us consider the following invertible matrix
\begin{equation}
M=\begin{pmatrix}
s_{01} & 0\\
\eta_1-s_{10} & \eta_2
\end{pmatrix}.
\end{equation}
Using the translation $(u_n,v_n)^T=M(X_n,Y_n)^T$, the model becomes
\begin{eqnarray*}
X_{n+1} &=& \eta_1X_n+\eta_2Y_n+\psi(X_n,Y_n)+O\left((|X_n|+|Y_n|)^3\right),\nonumber\\
Y_{n+1} &=& \eta_2X_n+\eta_1Y_n+\phi(X_n,Y_n)+O\left((|X_n|+|Y_n|)^3\right),
\end{eqnarray*}
where
\begin{eqnarray*}
\psi(X_n,Y_n) &=& \frac{1}{s_{01}}[\{s_{20}s^2_{01} + s_{11}s_{01}(\eta_1 - s_{10}) + s_{02}(\eta_1 - s_{10})^2\}X^2_n\\
&& +\{s_{11}s_{01}\eta_2 + 2\eta_2 s_{02}(\eta_1 - s_{10})\}X_nY_n\\
&& + s_{02}\eta_2^2Y^2_n + \{s_{30}s^3_{01} + s_{21}s^2_{01}(\eta_1 - s_{10}) + s_{12}s_{01}(\eta_1 - s_{10})^2\\
&& + s_{03}(\eta_1 - s_{10})^3\}X^3_n\\
&& + \{\eta_2s_{21}s^2_{01} + 2\eta_2s_{12}s_{01}(\eta_1 - s_{10})+3\eta_2s_{03}(\eta_1 - s_{10})^2\}X^2_nY_n\\
&& +\{\eta_2^2s_{12}s_{01} + 3\eta_2^2s_{03}(\eta_1 - s_{10})\}X_nY^2_n + \eta_2^3s_{03}Y^3_n],\\
\phi(X_n, Y_n) &=& \frac{1}{\eta_2s_{01}}[\{s^2_{01}s_{20}(s_{10} - \eta_1) + s^3_{01}t_{20} + s_{01}(\eta_1 - s_{10})(s_{11}(s_{10} - \eta_1)\\
&& +s_{01}t_{11}) + s_{02}(s_{10} - \eta_1)^3\}X^2_n + \{\eta_2s_{01}(s_{11}(s_{10} - \eta_1) + s_{01}t_{11})\\
&& -2\eta_2s_{02}(s_{10} - \eta_1)^2\}X_nY_n + \eta_2^2s_{02}(s_{10} - \eta_1)Y^2_n\\
&& +\{s_{30}s^3_{01}(s_{10} - \eta_1) + t_{30}s^4_{01} + s^2_{01}(\eta_1 - s_{10})(s_{21}(s_{10} - \eta_1) + s_{01}t_{21})\\
&& -s_{12}s_{01}(\eta_1 - s_{10})^3 - s_{03}(\eta_1 - s_{10})^4\}X^3_n + \{\eta_2s^2_{01}(s_{21}(s_{10} - \eta_1) + s_{01}t_{21})\\
&& -2\eta_2s_{01}s_{12}(s_{10} - \eta_1)^2 + 3\eta_2s_{03}(s_{10} - \eta_1)^3\}X^2_nY_n\\
&& +\{\eta_2^2s_{01}s_{12}(s_{10} - \eta_1) - 3\eta_2^2s_{03}(s_{10} - \eta_1)^2\}X_nY^2_n\\
&& +\eta_2^3s_{03}(s_{10} - \eta_1)Y^3_n].
\end{eqnarray*}
To ensure the occurrence of NS-bifurcation, we need to have the following quantity
\begin{equation}
\sigma^*=-\text{Re}\left[\frac{(1-2\lambda)\bar{\lambda}^2}{1-\lambda}\gamma_{11}\gamma_{20}\right]-\frac{1}{2}|\gamma_{11}|^2-|\gamma_{02}|^2+\text{Re}(\bar{\lambda}\gamma_{21})
\end{equation}
not equal to zero, where
\begin{eqnarray*}
\gamma_{20}&=& \frac{1}{8}[(\psi_{X_nX_n}-\psi_{Y_nY_n}+2\phi_{X_nY_n})+i(\phi_{X_nX_n}-\phi_{Y_nY_n}-2\psi_{X_nY_n})],\\
\gamma_{11}&=& \frac{1}{4}[(\psi_{X_nX_n}+\psi_{Y_nY_n})+i(\phi_{X_nX_n}+\phi_{Y_nY_n})],\\
\gamma_{02}&=& \frac{1}{8}[(\psi_{X_nX_n}-\psi_{Y_nY_n}-2\phi_{X_nY_n})+i(\phi_{X_nX_n}-\phi_{Y_nY_n}+2\psi_{X_nY_n})],\\
\gamma_{21}&=& \frac{1}{16}[(\psi_{X_nX_nX_n}+\psi_{X_nY_nY_n}+\phi_{X_nX_nY_n}+\phi_{Y_nY_nY_n})\\
& & \quad +i(\phi_{X_nX_nX_n}+\phi_{X_nY_nY_n}-\psi_{X_nX_nY_n}-\psi_{Y_nY_nY_n})].
\end{eqnarray*}
Hence, we derive the following conclusion based on the preceding analysis and the Neimark-Sacker theorem.
\begin{theorem}
When conditions \eqref{bif_cond_1} and \eqref{bif_cond_2} are satisfied with $\sigma^*=0$, the model \eqref{sys5} exhibits a Hopf bifurcation at the equilibrium point $E_+(x_+, y_+)$ as the parameter $\beta$ varies within a small neighborhood of $\Omega$. Additionally, if $\sigma^*$ is less than $0$, an attracting invariant closed curve emerges from the fixed point $E_+(x_+, y_+)$ when $\beta>\beta_{NS}$, while if $\sigma^*$ is greater than 0, a repelling invariant closed curve arises from the fixed point $E_+(x_+, y_+)$ when $\beta < \beta_{NS}$.
\end{theorem}

\section{Numerical simulations}\label{sec_num}
In this section, we conduct a numerical analysis of system \eqref{sys5}. Here, we investigate the asymptotic behaviour of the system and describe how it behaves with respect to changes in key system parameters. Our observations reveal that the system undergoes Neimark-Sacker bifurcation for the parameters \(s,\; w,\; \alpha,\; \beta\;\&\;\theta\). We provide bifurcation diagrams alongside phase portraits illustrating the system dynamics. The considered parameter values for the numerical simulation are \[r=1.293,\; K=4,\; p=0.05,\; q=0.5,\; a=0.7,\; h=3,\; b=0.6\;\; \& \;\; c=0.168.\] After applying the transformations used to obtain system \eqref{sys5}, these transform into the following parameter values: \[s=0.0125,\; w=0.125,\; \alpha=8.4,\; \beta=1.3\;\; \& \;\; \theta=0.13.\] The fixed points are evaluated to be $(0,0)$, $(0.0125,0)$, $(1,0)$ and $E_+=(0.625,1.91406)$. The positive fixed point is denoted by a red dot $({\color{red}\Large \cdot})$ in the phase portraits.
\begin{figure}[H]
 \subfloat[bifurcation diagrams]{%
  \includegraphics[width=0.33\textwidth, height=0.3\textwidth]{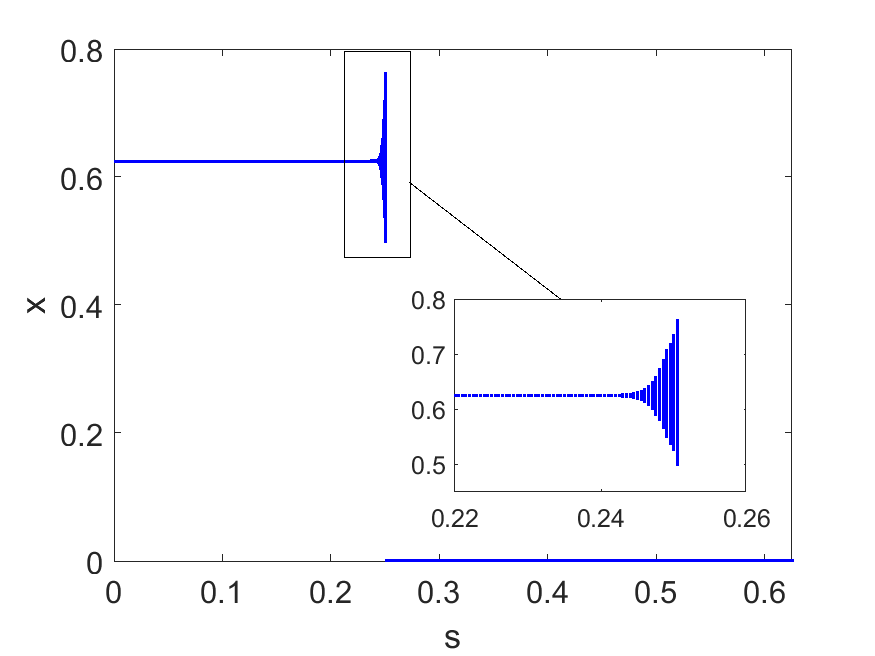}
  \includegraphics[width=0.33\textwidth, height=0.3\textwidth]{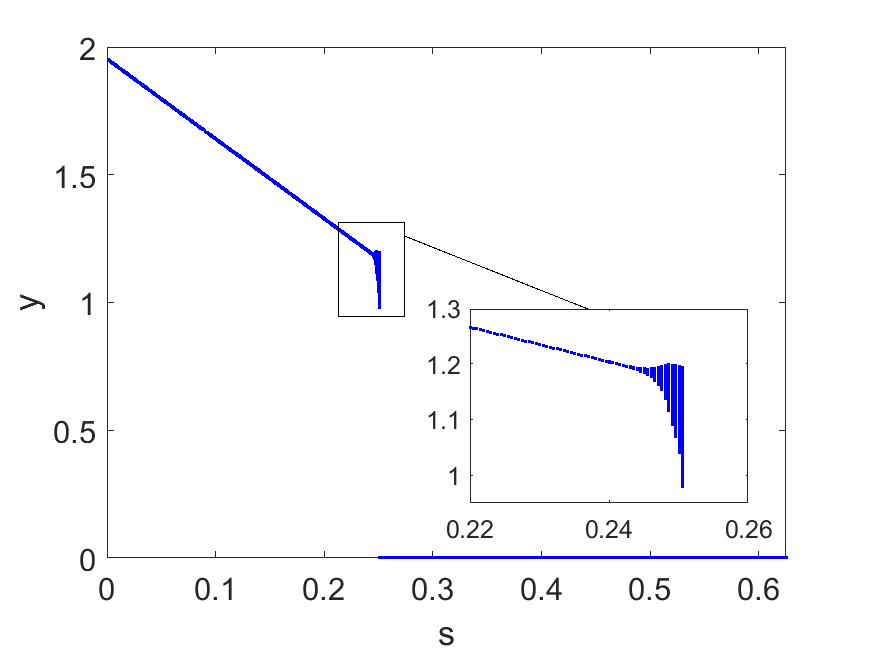}}
  \subfloat[phase portrait for $s<s_{NS}$]{%
  \includegraphics[width=0.33\textwidth, height=0.3\textwidth]{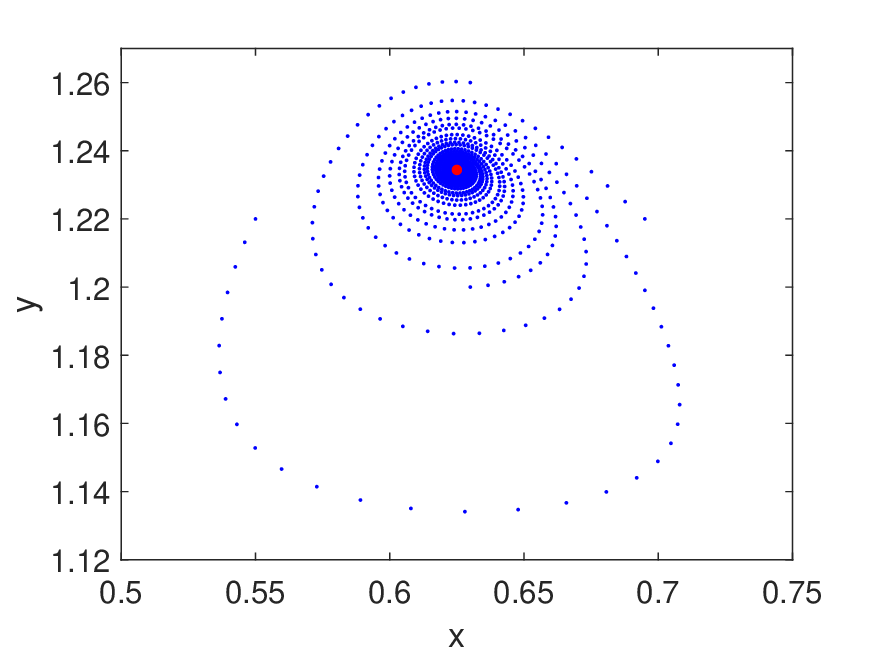}}\\
 \subfloat[phase portrait for $s=s_{NS}$]{%
  \includegraphics[width=0.33\textwidth, height=0.3\textwidth]{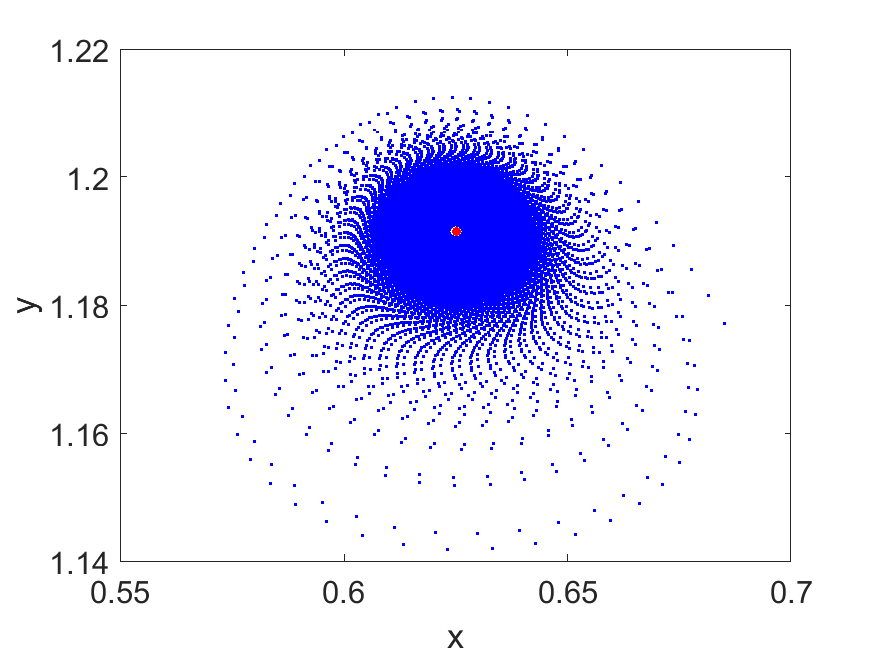}}
 \subfloat[phase portrait for $s>s_{NS}$]{%
  \includegraphics[width=0.33\textwidth, height=0.3\textwidth]{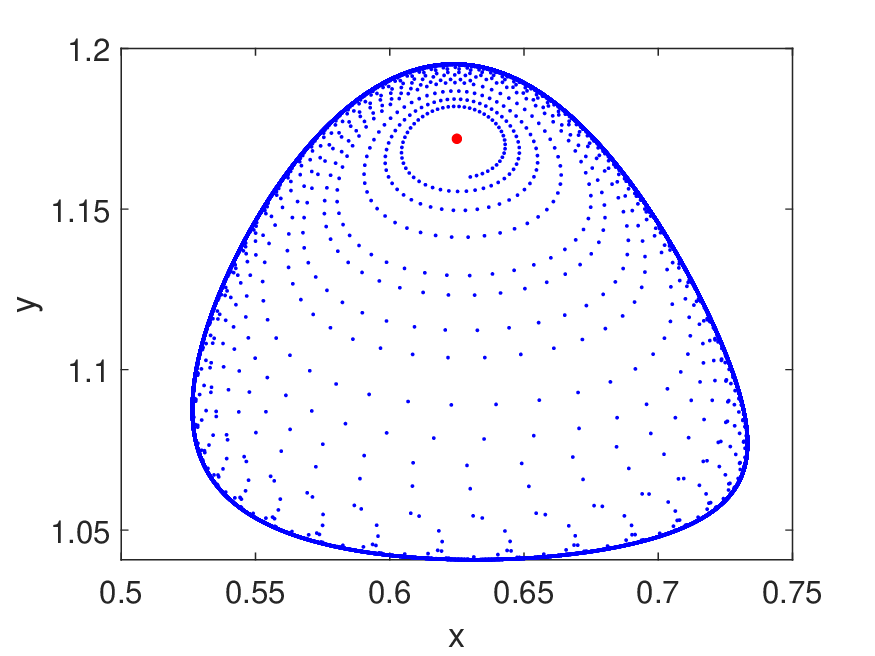}}
 \subfloat[phase portrait for $s>s_{th}$]{%
  \includegraphics[width=0.33\textwidth, height=0.3\textwidth]{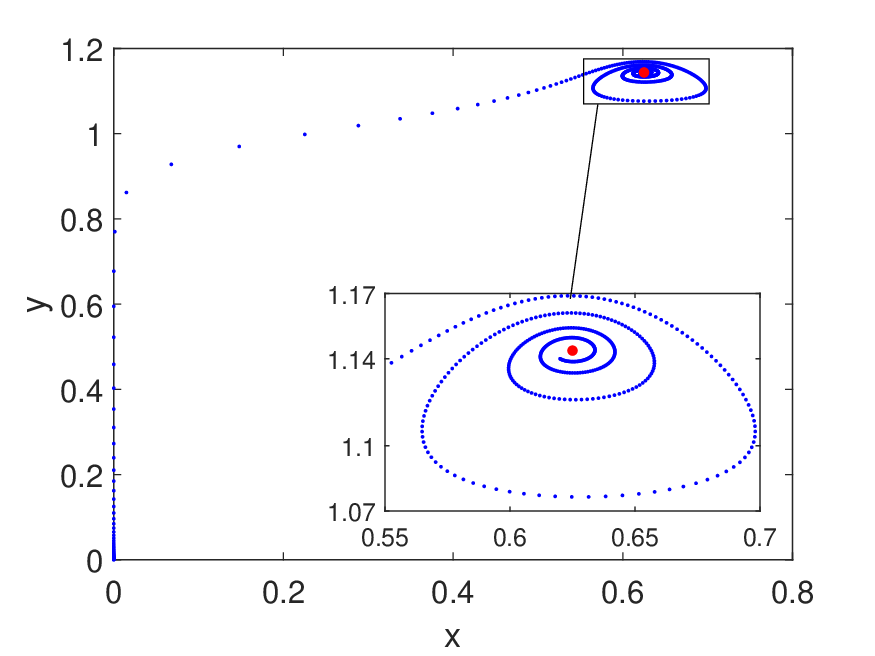}} 
 \caption{{\bf a.~}System \eqref{sys5} undergoes a Neimark-Sacker bifurcation when $s$ crosses the bifurcation point $s_{NS}=0.243716$ for $w=0.125$, $\alpha=8.4$, $\beta=1.3$ \& $\theta=0.13$ {~\bf b.~}$E_+$ is stable (sink) for $s=0.23<s_{NS}$ {~\bf c.~}$E_+$ is stable for $s=s_{NS}=0.243716$ {~\bf d.~}$E_+$ is unstable (source) \& a stable limit cycle emerges when $s=0.25>s_{NS}$ {~\bf e.~}both species become extinct for $s>s_{th}=0.259$}
 \label{s_bif}
\end{figure}
Figure \ref{s_bif} shows the occurrence of Neimark-Sacker bifurcation in system \eqref{sys5} with respect to the strong Allee threshold value $(s)$. The bifurcation point is obtained to be $s_{NS}=0.243716$. For values of $s$ less than $s_{NS}$, the positive fixed point $E_+$ becomes locally asymptotically stable (sink), whereas, for values of $s$ greater than $s_{NS}$, the system possesses a stable limit cycle until $s$ reaches the threshold value $s_{th}=0.259$. In this case, the fixed point becomes unstable (source). When $s$ crosses the threshold value $s_{th}$, both species become extinct since the origin becomes the only stable state in the phase space.
\begin{figure}[H]
 \subfloat[bifurcation diagrams]{%
  \includegraphics[width=0.33\textwidth, height=0.3\textwidth]{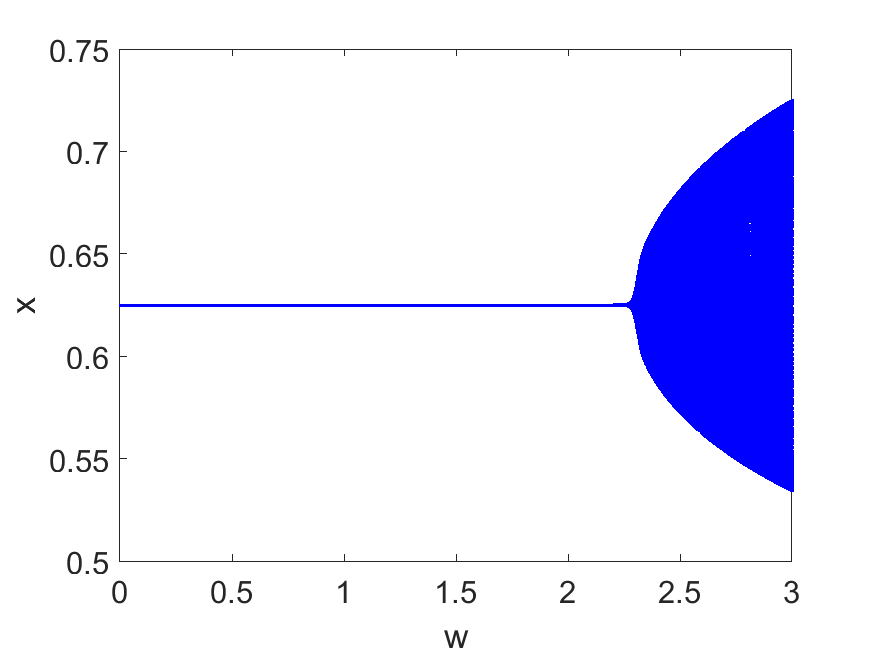}
  \includegraphics[width=0.33\textwidth, height=0.3\textwidth]{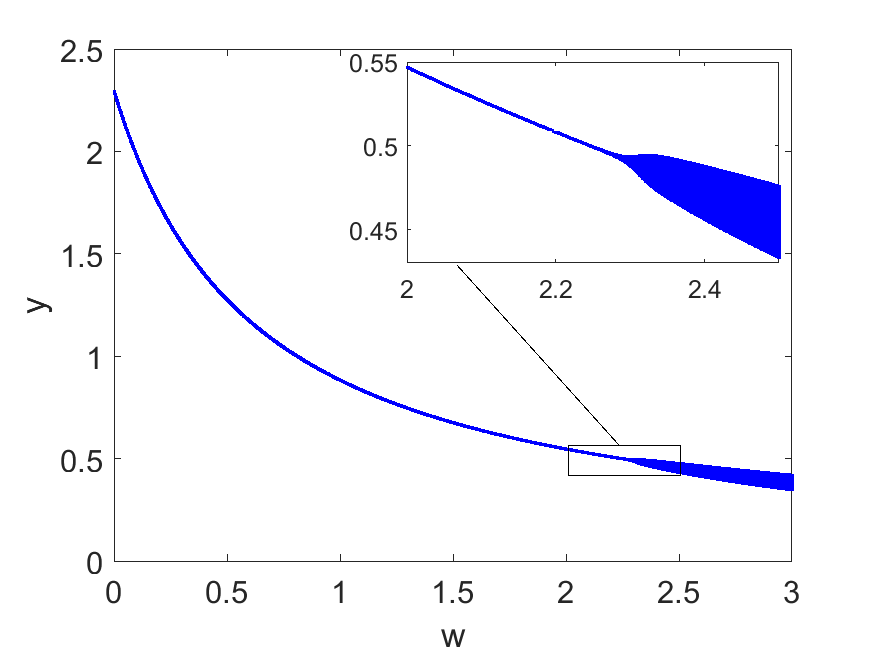}}\\
  \subfloat[phase portrait for $w<w_{NS}$]{%
  \includegraphics[width=0.33\textwidth, height=0.3\textwidth]{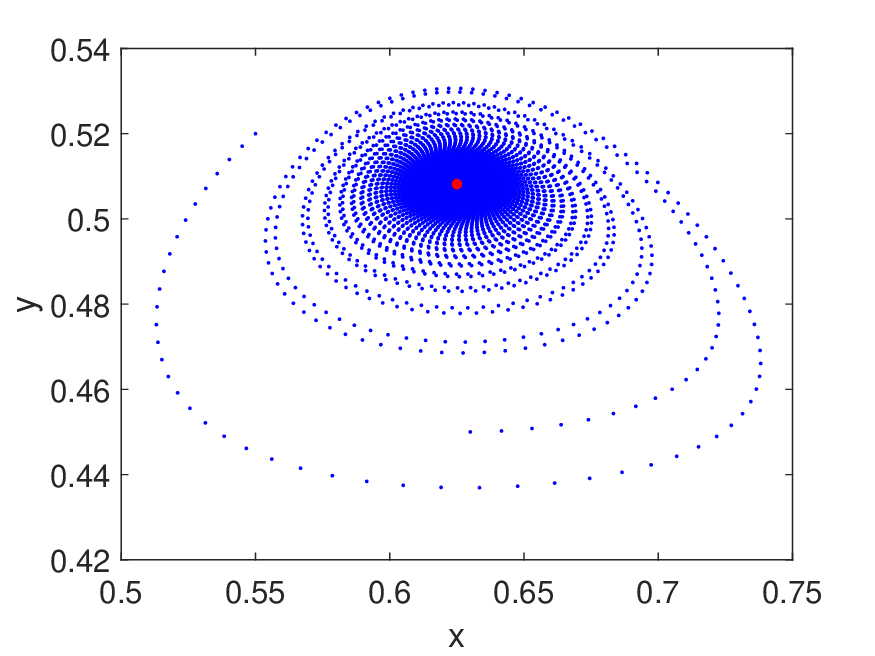}}
 \subfloat[phase portrait for $w=w_{NS}$]{%
  \includegraphics[width=0.33\textwidth, height=0.3\textwidth]{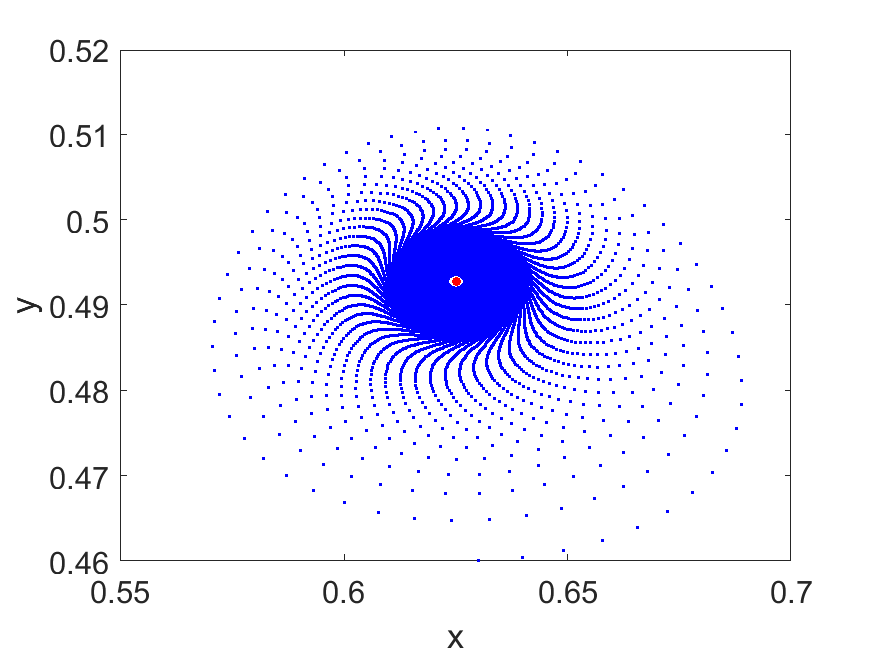}}
 \subfloat[phase portrait for $w>w_{NS}$]{%
  \includegraphics[width=0.33\textwidth, height=0.3\textwidth]{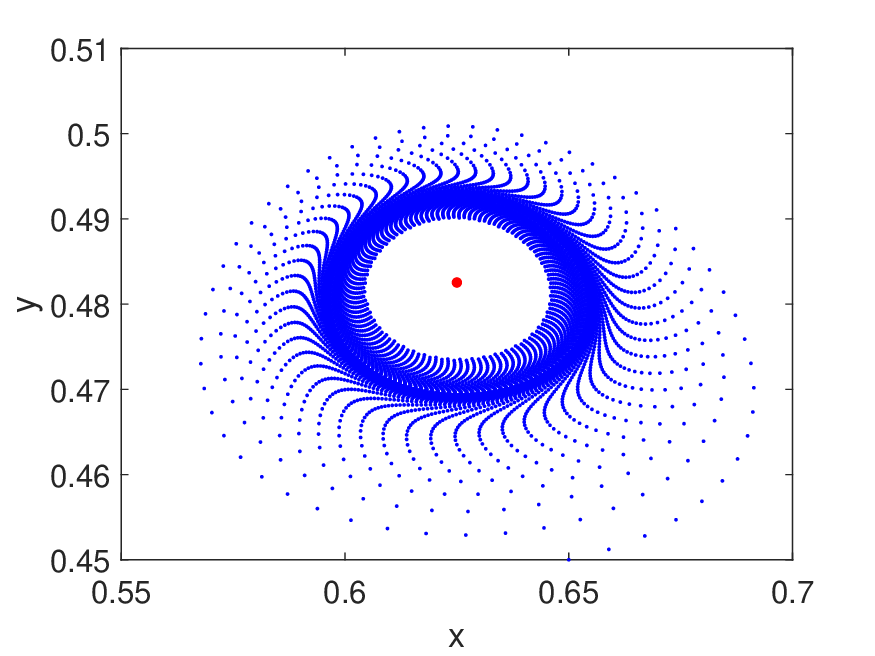}}
 \caption{{\bf a.~}System \eqref{sys5} undergoes a Neimark-Sacker bifurcation when $w$ crosses the bifurcation point $w_{NS}$ for $s=0.0125$, $\alpha=8.4$, $\beta=1.3$ \& $\theta=0.13$ {~\bf b.~} $E_+$ is stable for $w=2.2<w_{NS}$ {~\bf c.~} $E_+$ is stable at the bifurcation point $w=w_{NS}=2.288189$ {~\bf d.~} $E_+$ becomes unstable and a stable limit cycle emerges for $w=2.35>w_{NS}$}
 \label{w_bif}
\end{figure}
In Fig. \ref{w_bif}, the bifurcation diagrams illustrate the occurrence of Neimark-Sacker bifurcation in system \eqref{sys5} as the weak Allee constant $(w)$ crosses the bifurcation point $w_{NS}=2.288189$. The phase portraits demonstrate the emergence of a stable limit cycle in the phase space for $w>w_{NS}$. This implies that, in the long term, the population tends to stabilize at the co-existent fixed point for $w<w_{NS}$, whereas the densities tend to oscillate periodically when $w>w_{NS}$.

Figs. \ref{alpha_bif}, \ref{beta_bif}, and \ref{theta_bif} exhibit a similar phenomenon to Fig. \ref{s_bif}. Neimark-Sacker bifurcation occurs in the system as the parameters $\alpha$, $\beta$, and $\theta$ vary and cross the bifurcation points $\alpha_{NS}$, $\beta_{NS}$, and $\theta_{NS}$, respectively. These parameters also possess threshold values beyond which the populations cease to exist. The threshold values are $\alpha_{th}$ and $\theta_{th}$ for $\alpha$ and $\theta$ below which the populations collapse, and $s_{th}$ and $\beta_{th}$ for $s$ and $\beta$ above which species extinction occurs. This phenomenon can be attributed to the presence of multiple Allee effects within the prey population. For example, in Fig. \ref{alpha_bif}, as $\alpha$ decreases from $\alpha_{NS}=8.048817$, the amplitude of the limit cycle increases. Despite a critical population density of $s=0.0125$ in the simulation, Fig. \ref{alpha_bif_xy} shows that the minimum value attained by $x$ in limit cycles exceeds $0.2$. This indicates that population extinction can happen even when the density surpasses the critical level significantly. Hence, it suggests that multiple Allee effects can notably elevate the critical threshold level.
\begin{figure}[H]
 \subfloat[bifurcation diagrams]{%
  \includegraphics[width=0.33\textwidth, height=0.3\textwidth]{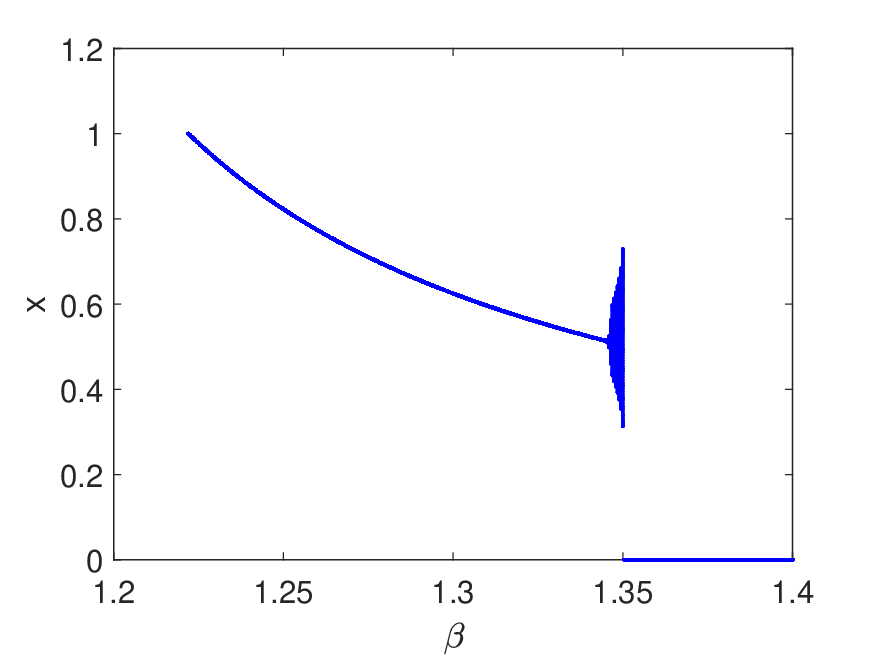}
  \includegraphics[width=0.33\textwidth, height=0.3\textwidth]{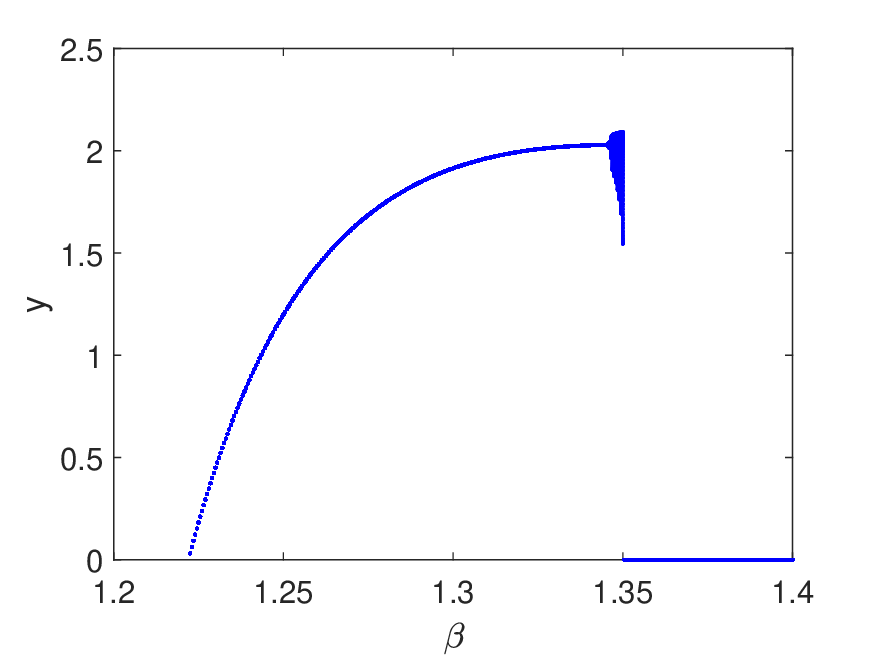}}
 \subfloat[phase portrait for $\beta<\beta_{NS}$]{%
  \includegraphics[width=0.3\textwidth, height=0.3\textwidth]{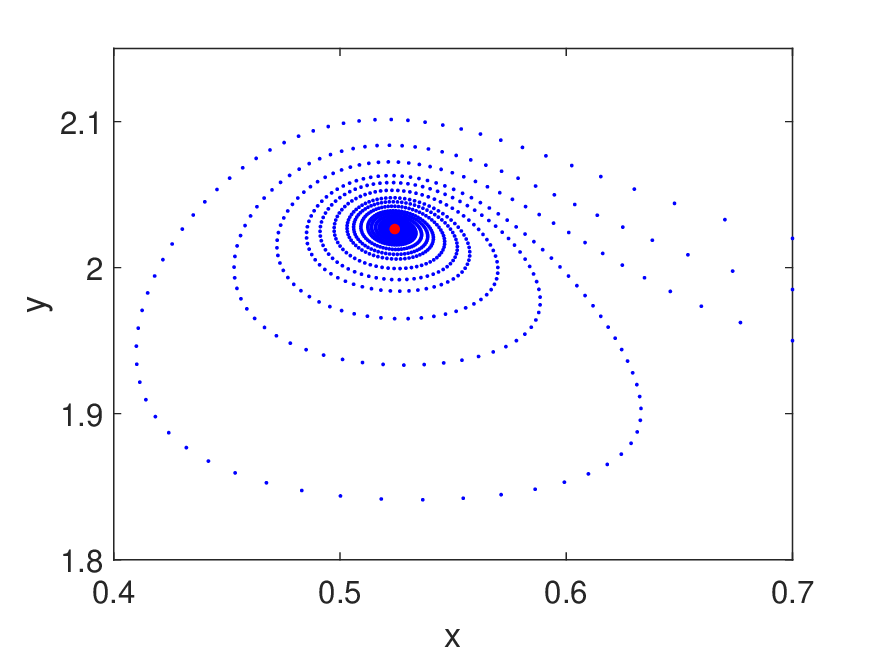}}\\
 \subfloat[phase portrait for $\beta=\beta_{NS}$]{%
  \includegraphics[width=0.33\textwidth, height=0.3\textwidth]{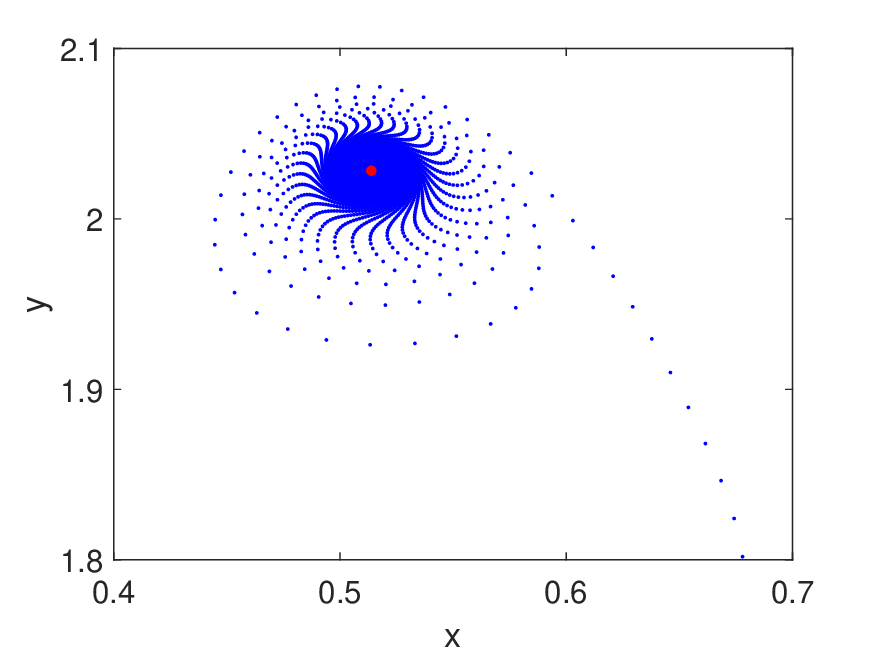}}
 \subfloat[phase portrait for $\beta>\beta_{NS}$]{%
  \includegraphics[width=0.33\textwidth, height=0.3\textwidth]{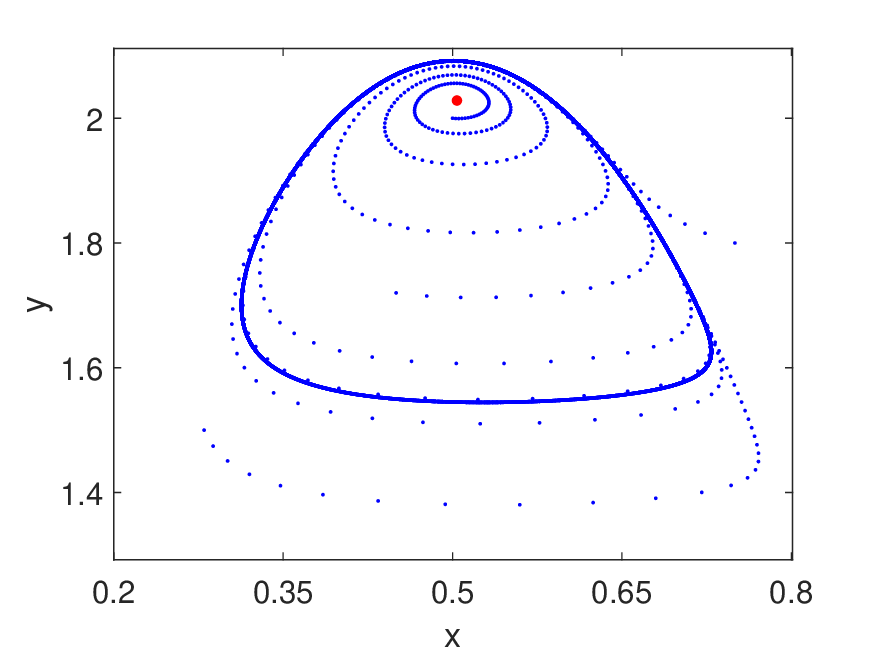}}
 \subfloat[phase portrait for $\beta>\beta_{th}$]{%
  \includegraphics[width=0.33\textwidth, height=0.3\textwidth]{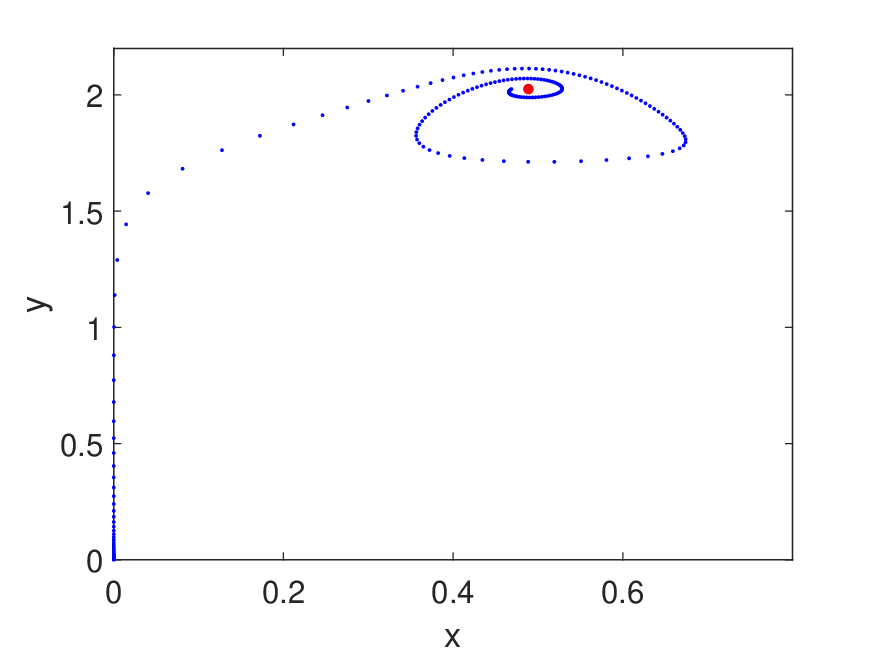}} 
 \caption{{\bf a.~}System \eqref{sys5} undergoes a Neimark-Sacker bifurcation when $\beta$ crosses the bifurcation point $\beta_{NS}$ for $s=0.0125$, $w=0.125$, $\alpha=8.4$, \& $\theta=0.13$ {~\bf b.~} $E_+$ is asymptotically stable for $\beta=1.34<\beta_{NS}$ {~\bf c.~} $E_+$ is stable at $\beta=\beta_{NS}=1.344995$ {~\bf d.~} $E_+$ is unstable and a stable limit cycle emerges for $\beta=1.35>\beta_{NS}$ {~\bf e.~} $(0,0)$ is the only stable in the phase space for $\beta>\beta_{th}$}
 \label{beta_bif}
\end{figure}

\begin{figure}[H]
 \subfloat[bifurcation diagrams\label{alpha_bif_xy}]{%
  \includegraphics[width=0.33\textwidth, height=0.3\textwidth]{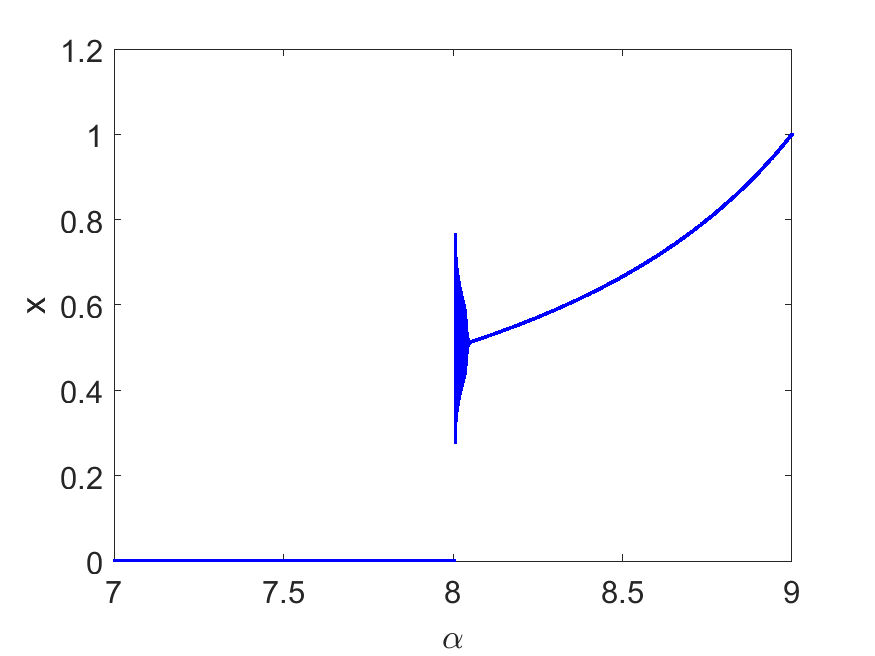}
  \includegraphics[width=0.33\textwidth, height=0.3\textwidth]{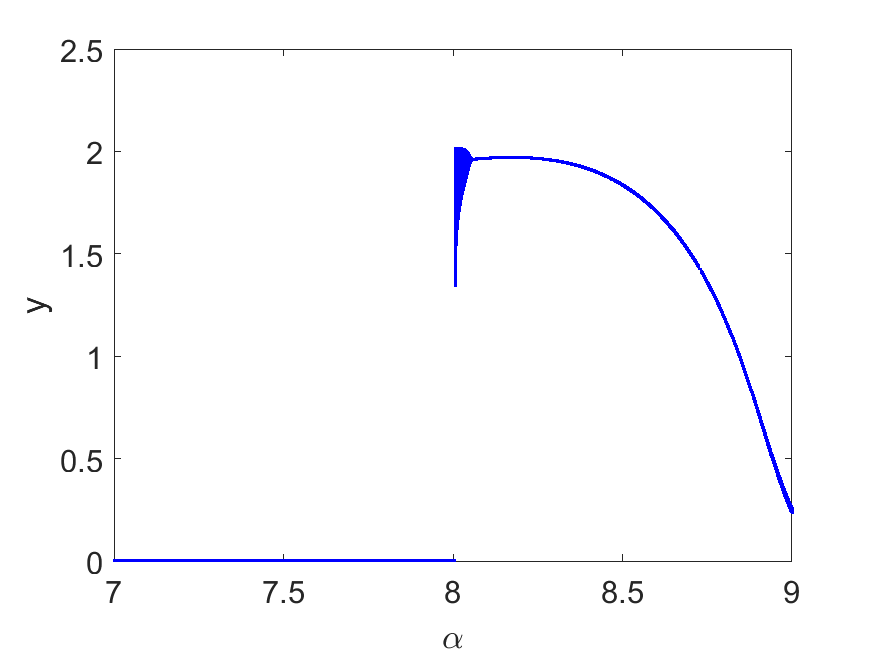}}
  \subfloat[phase portrait for $\alpha<\alpha_{th}$\label{alpha_b}]{%
  \includegraphics[width=0.33\textwidth, height=0.3\textwidth]{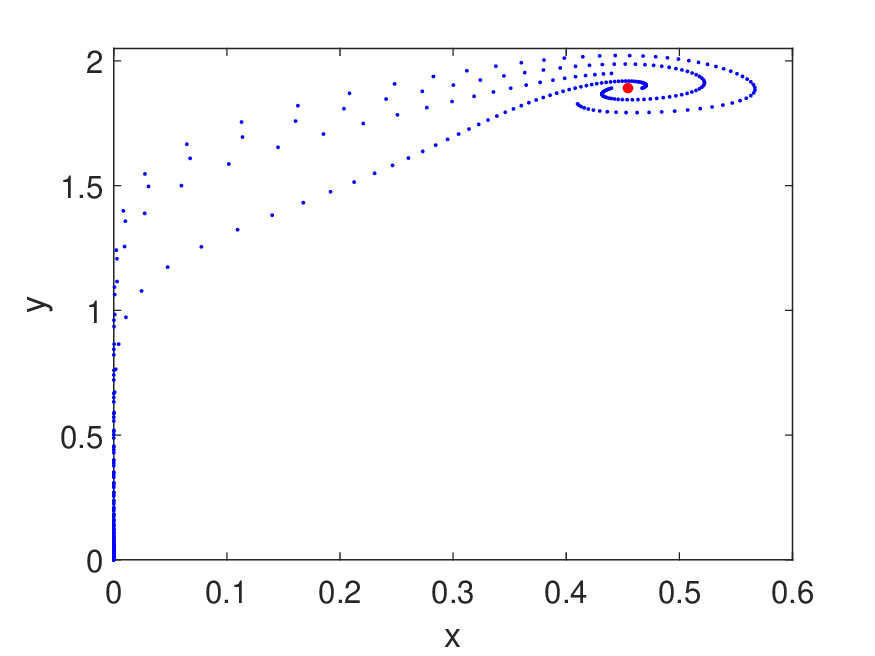}}\\
 \subfloat[phase portrait for $\alpha_{th}<\alpha<\alpha_{NS}$]{%
  \includegraphics[width=0.33\textwidth, height=0.3\textwidth]{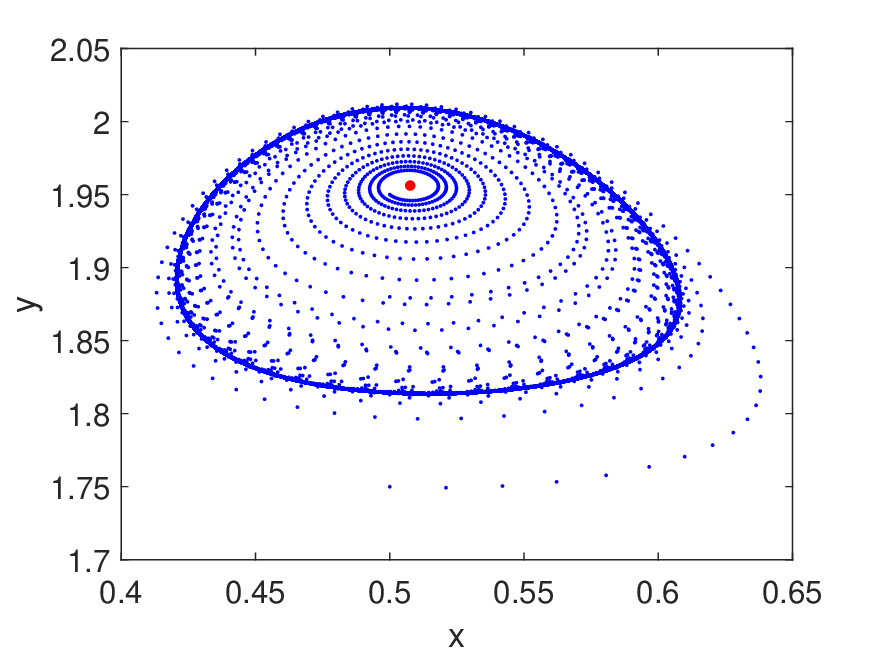}}
 \subfloat[phase portrait for $\alpha=\alpha_{NS}$]{%
  \includegraphics[width=0.33\textwidth, height=0.3\textwidth]{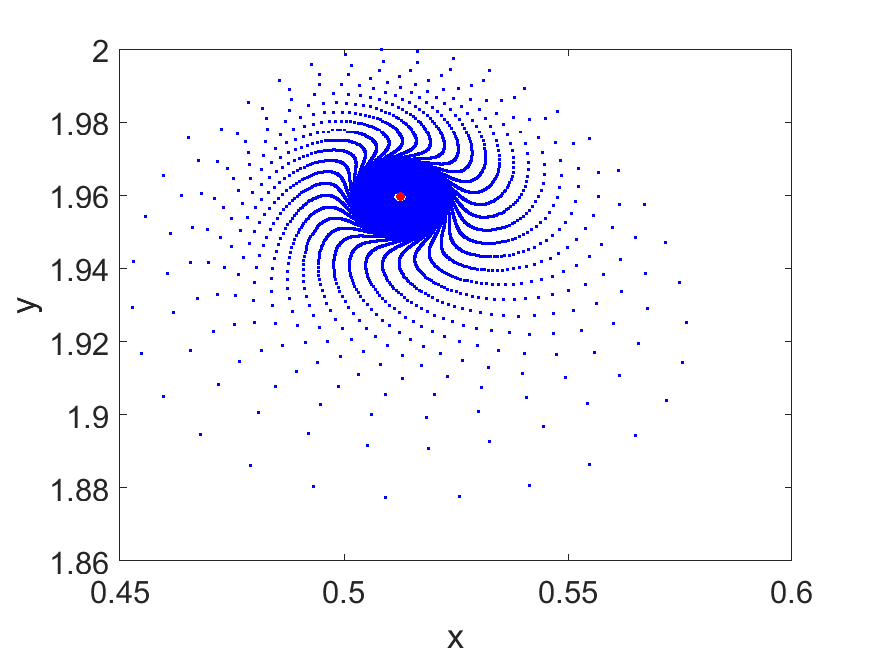}} 
 \subfloat[phase portrait for $\alpha>\alpha_{NS}$]{%
  \includegraphics[width=0.33\textwidth, height=0.3\textwidth]{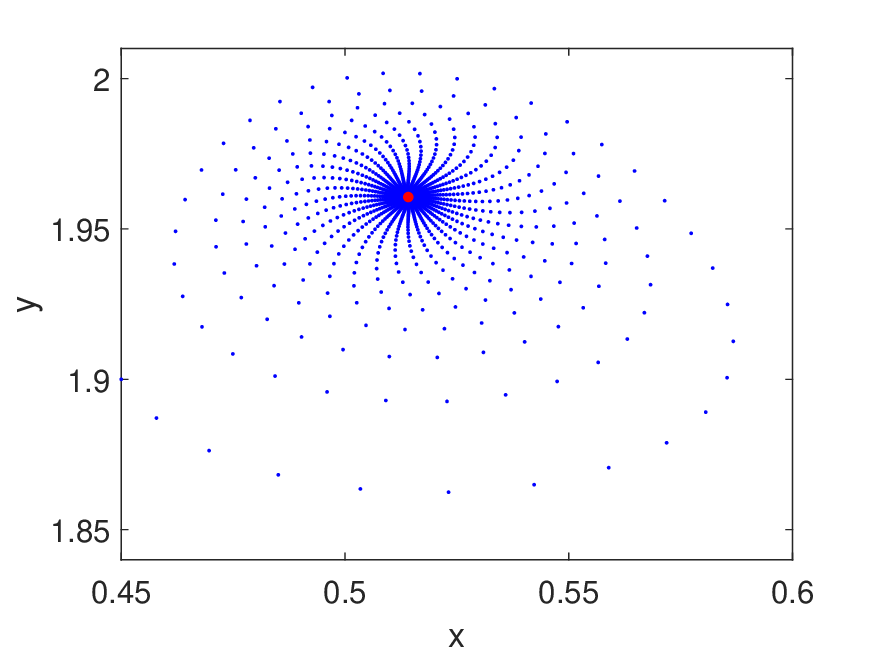}} 
 \caption{{\bf a.~}System \eqref{sys5} undergoes a Neimark-Sacker bifurcation when $\alpha$ crosses the bifurcation point $\alpha_{NS}$ for $s=0.0125$, $w=0.125$, $\beta=1.3$ \& $\theta=0.13$ {~\bf b.~} $(0,0)$ is the only stable state in the phase space for $\alpha=7.8<\alpha_{th}$ {~\bf c.~} $E_+$ is unstable and a stable limit cycle exists for $\alpha_{th}<\alpha=8.03<\alpha_{NS}$ {~\bf d.~} $E_+$ is stable at $\alpha=\alpha_{NS}=8.048817$ {~\bf e.~} $E_+$ is stable for $\alpha=8.06>\alpha_{NS}$}
 \label{alpha_bif}
\end{figure}

\begin{figure}[H]
 \subfloat[bifurcation diagrams]{%
  \includegraphics[width=0.33\textwidth, height=0.3\textwidth]{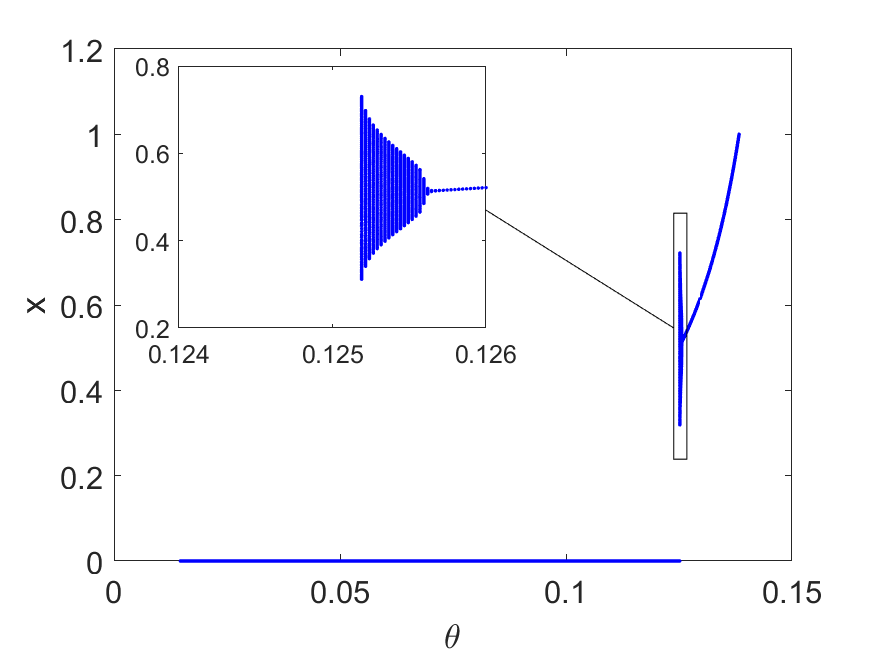}
  \includegraphics[width=0.33\textwidth, height=0.3\textwidth]{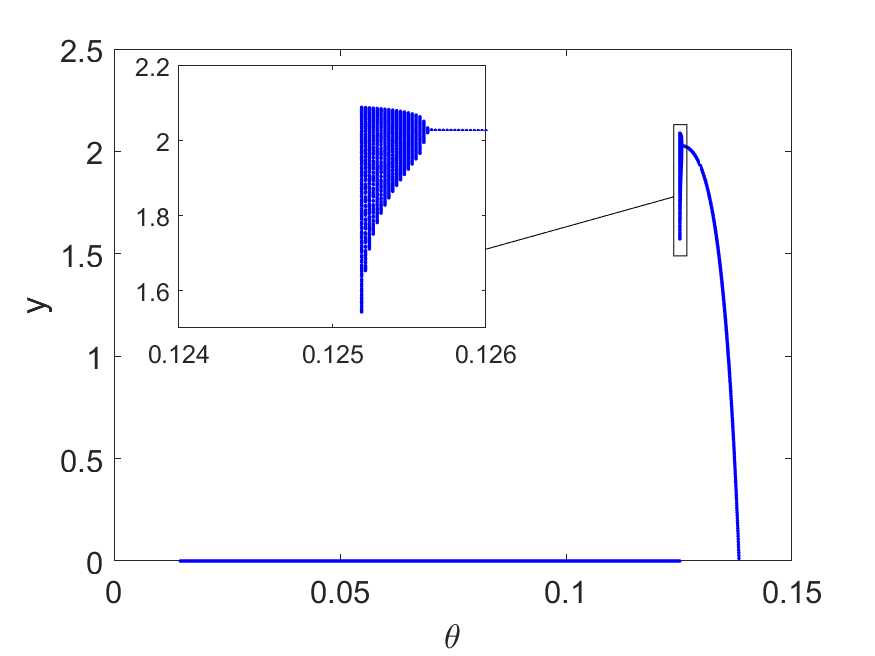}}
  \subfloat[phase portrait for $\theta<\theta_{th}$]{%
  \includegraphics[width=0.33\textwidth, height=0.3\textwidth]{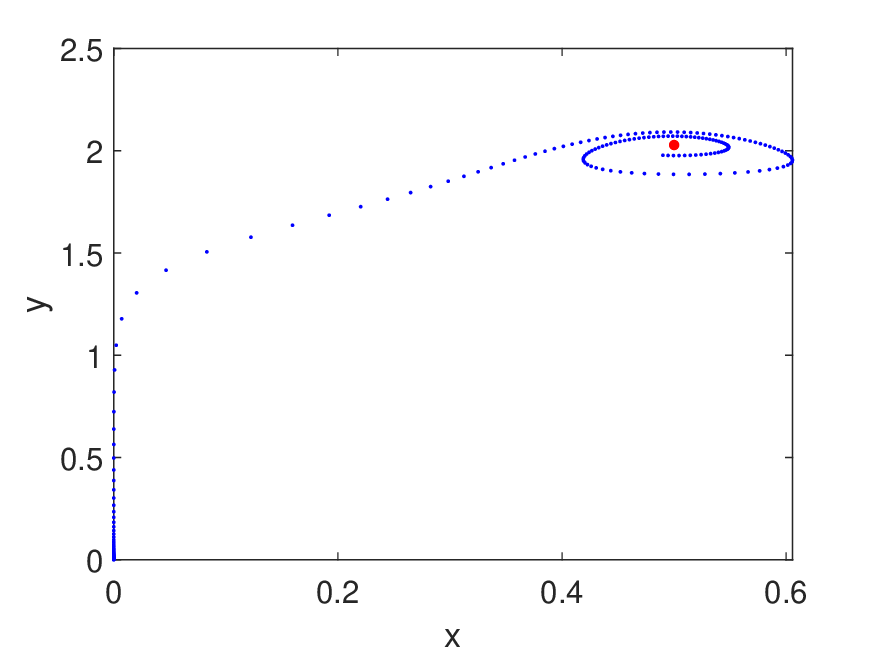}}\\
 \subfloat[phase portrait for $\theta_{th}<\theta<\theta_{NS}$]{%
  \includegraphics[width=0.33\textwidth, height=0.3\textwidth]{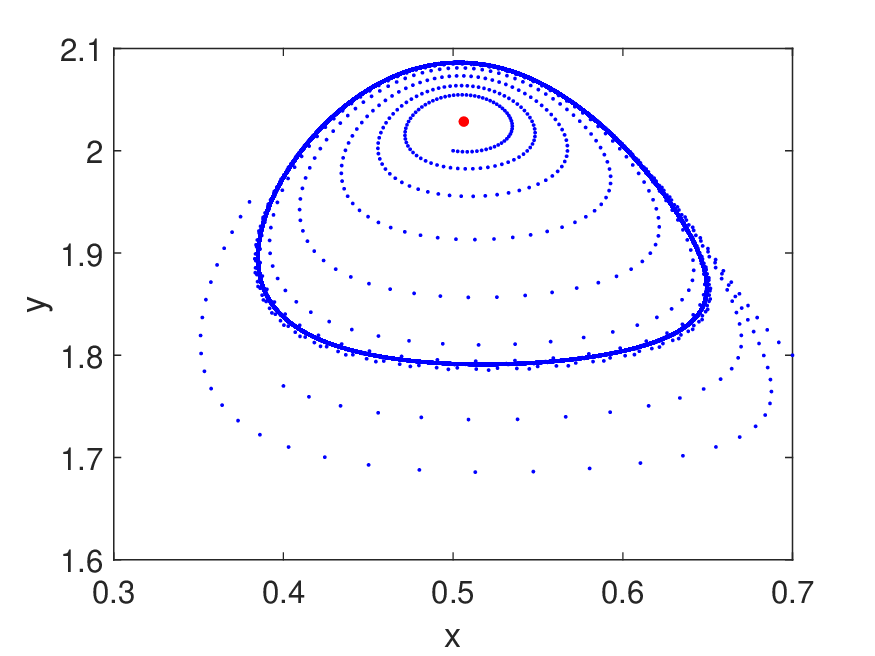}}
 \subfloat[phase portrait for $\theta=\theta_{NS}$]{%
  \includegraphics[width=0.33\textwidth, height=0.3\textwidth]{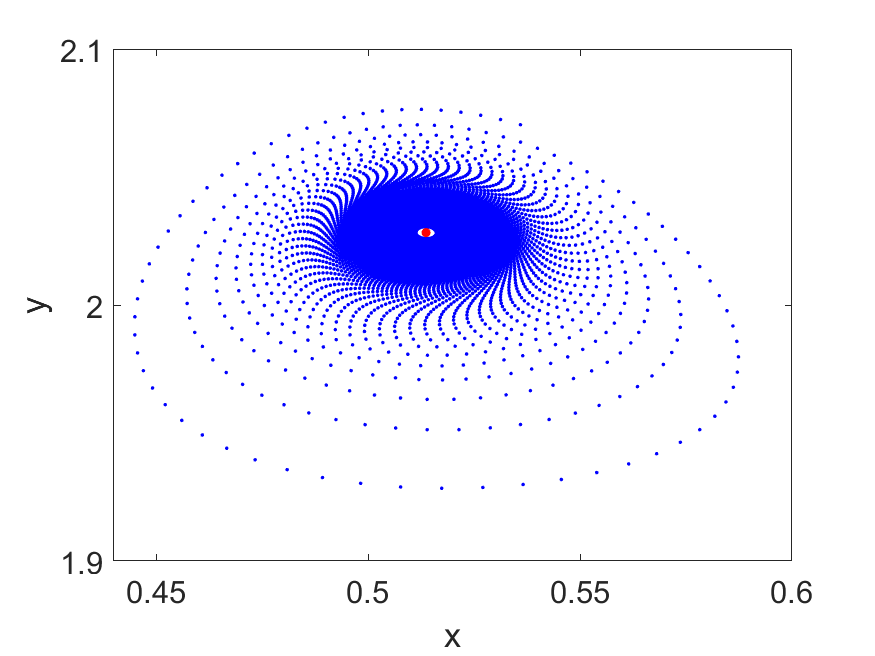}}
  \subfloat[phase portrait for $\theta>\theta_{NS}$]{%
  \includegraphics[width=0.33\textwidth, height=0.3\textwidth]{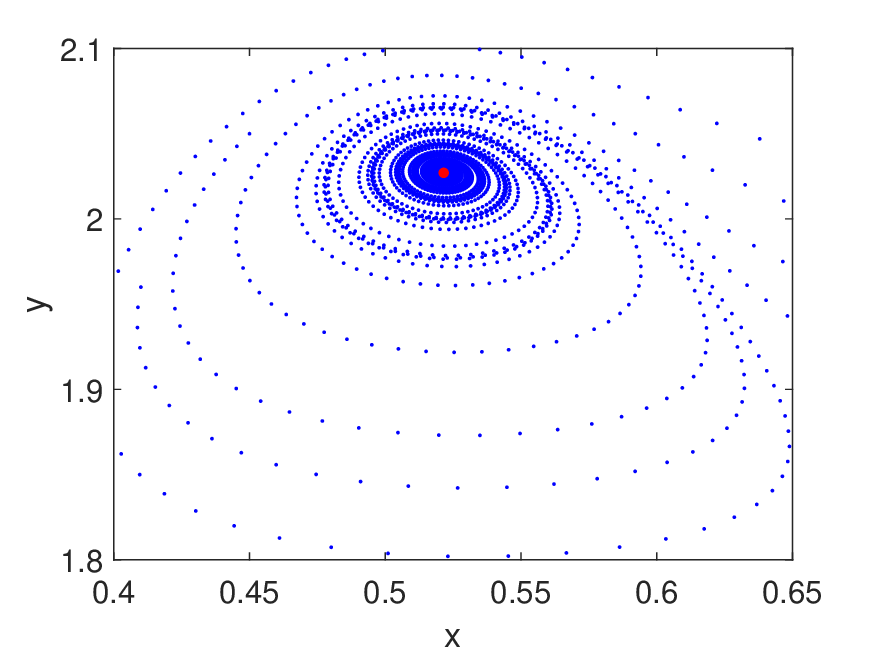}} 
 \caption{{\bf a.~}System \eqref{sys5} undergoes a Neimark-Sacker bifurcation when $s$ crosses the bifurcation point $s_{NS}$ for $s=0.0125$, $w=0.125$, $\alpha=8.4$ \& $\beta=1.3$ {~\bf b.~} $(0,0)$ is the only stable state for $\theta<\theta_{th}$ {~\bf c.~} $E_+$ is unstable and a stable limit cycle exists for $\theta_{th}<\theta=0.1253<\theta_{NS}$ {~\bf d.~} $E_+$ becomes stable for $\theta=\theta_{NS}=0.125642$ {~\bf e.~}$E_+$ is stable for $\theta=0.126>\theta_{NS}$}
 \label{theta_bif}
\end{figure}
\section{Conclusion}\label{sec_con}
In this article, we have proposed a two-dimensional discrete predator-prey model, which is suitable for analysing the dynamics of nonoverlapping generations of populations. The prey growth is limited by the carrying capacity of the environment and is also subjected to two component Allee effects; one is a weak Allee effect, while the other is a strong Allee effect. As mentioned earlier, the weak Allee effect may be associated with the measure of the non-fertile population in the species, and the strong Allee effect may be associated with the mate-finding Allee effect or reduced protection from predators during low population density.

Discussing the dynamical aspects of the populations, we found four possible fixed points when there is a strong Allee effect $(s>0)$ present in the system and three fixed points when both Allee effects are weak $(s<0)$, including the trivial fixed point $(0,0)$, one prey-only axial fixed point $(1,0)$ and one conditionally existent positive fixed point. The local stability of all fixed points is discussed, and the sufficient conditions for different stability natures attained by the fixed points are obtained. Also, we explored the region of stability for each fixed point in a two-dimensional parametric plane $(sw-plane)$. We observed that the system parameters undergo Neimark-Sacker bifurcations, which we illustrated numerically for all parameters. For $\beta$, we derived the normal form for the Neimark-Sacker bifurcation, using which the direction, amplitude and other characteristics of the emerging limit cycles can be obtained.

A notable characteristic of a system with a strong Allee effect in prey populations is its tendency to drive the populations to extinction when the prey falls below a critical level. However, numerical simulations reveal that population extinction can occur even when densities are well above the critical thresholds. This phenomenon arises due to the presence of multiple Allee effects in the prey population. For instance, in Fig. \ref{alpha_bif}, the amplitude of the limit cycle increases for decreasing values of $\alpha$ starting from $\alpha_{NS}=8.048817$. Despite a critical population density of $s=0.0125$ in the simulation, Fig. \ref{alpha_bif_xy} illustrates that the minimum value of $x$ in a limit cycle exceeds $0.2$. As a result, population extinction occurs even when the density is significantly higher than the critical level, leading to the conclusion that multiple Allee effects can substantially raise the critical threshold level.

\section*{Funding}
Not applicable.
\section*{Conflict of interest/Competing interests}
The authors declare that they have no conflict of interest.
\section*{Availability of data and material (data transparency)}
Not applicable
\section*{Code availability}
Not applicable.

\end{document}